\documentclass[12pt]{amsart}
\usepackage{amsmath,amstext,amscd,amssymb,amsthm}
  \usepackage{paralist}
  \usepackage{amssymb} %% added by the authors
  \usepackage{graphics} %% add this and next lines if pictures should be in esp format
  \usepackage{epsfig} %For pictures: screened artwork should be set up with an 85 or 100 line screen
 \usepackage[colorlinks=true]{hyperref}
\usepackage{tikz}
\usetikzlibrary{automata,positioning,arrows}

\numberwithin{equation}{section}
\newtheorem{theorem}{Theorem}[section]
\newtheorem{corollary}[theorem]{Corollary}
\newtheorem{lemma}[theorem]{Lemma}
\newtheorem{proposition}[theorem]{Proposition}

\newtheorem{question}{Question}
\newtheorem{definition}[theorem]{Definition}
\newtheorem{remark}{Remark}
\newcommand{\ep}{\varepsilon}

\newcommand{\bz}{\mathbb{Z}}
\newcommand{\bb}{\mathbb{F}_p}
\newcommand{\Iff}{if and only if }

\newcommand{\lb}{\{}
\newcommand{\rb}{\}}
\DeclareMathOperator{\Aut}{Aut}
\DeclareMathOperator{\Deg}{Deg}
\DeclareMathOperator{\Stab}{Stab}
\begin{document}

\title{On the lattice of subgroups of the lamplighter group}

\author{R. Grigorchuk}
\address{Department of Mathematics, Texas A\&M University, College Station, TX 77843, USA}
\email{grigorch@math.tamu.edu}
\author{R. Kravchenko}
\address{Department of Mathematics, University of Chicago, Chicago, IL 60637, USA}
\email{rkchenko@math.uchicago.edu}
\thanks{The first author was supported by Simons Foundation and Mittag-Leffler Institute.  Both authors were supported by the ERC starting grant GA 257110 ÒRaWGÓ}
%\thanks{\thinspace ${\hspace{-.45ex}}^\star$Department of Mathematics,
%Texas A\&M University, College Station, TX 77843, USA; \ts
%Email: \ts \texttt{grigorch@math.tamu.edu}}

%\thanks{\thi${\hspace{-.45ex}}^\dagger$Department of Mathematics,
%University of Chicago, Chicago, IL 60637, USA; \ts Email: \ts \texttt{rkchenko@math.uchicago.edu}}

%\date \today

%\subjclass[2000]{35Q30, 76F02}
%\keywords{}
\begin{abstract}
The paper is devoted to the study of the lattice of subgroups of the Lamplighter type groups and to the relative gradient rank.
\end{abstract}

\maketitle

\section{Introduction}
The group $\mathcal{L}$  called Lamplighter is defined as the wreath product $({\mathbb{Z}/2})\wr\mathbb{Z}$. This
group and its generalizations like wreath products of two groups and, more generally, iterated wreath products play important role not only in the group theory but also in theory of random walks  \cite{MR704539},   \cite{MR2276348},
\cite{BPS}, and other areas of mathematics. One of the most unusual properties of $\mathcal{L}$ is that the spectrum of Laplace operator  on Cayley  graph of $\mathcal{L}$ (constructed
using a special system of generators) is pure point spectrum and accordingly the spectral measure is discrete
\cite{MR1866850}. This result together with some other results which also involve the lamplighter type groups lead to the
solution in (\cite{MR1797748}, \cite{AUS}, \cite{GRA},  \cite{LW}, \cite{PSZ})  of different
versions of the Atiyah Problem \cite{MR0420729}.

The group $\mathcal{L}$ has a presentation
\[
\langle a,b | b^2, [b,a^jba^{-j}]\text{ for } j\geq 1\rangle,
\]
which is minimal  (hence the group is not finitely presented). $\mathcal{L}$ is a metabelian group of exponential growth
and is a self-similar group generated by $2$-state automaton over binary alphabet,  \cite{MR1866850},
\cite{MR1841755}, \cite{MR2467008}. Despite relatively simple algebraic structure, the group
$\mathcal{L}$ in fact is not so simple to study and there are many open questions related to it.

In this note we study a subgroup structure of $\mathcal{L}$ and of groups $\mathcal{L}_{n,p}=({\mathbb{Z}/p\mathbb{Z}})^n\wr\mathbb{Z},$ where $p$ is a prime.  
Groups $\mathcal{L}_{1,p}$ play especially important role in study of solvable groups.  For instance the result of Kropholler (\cite{Krop}) tells that the groups of finite rank in the class of finitely generated solvable groups are characterized as those groups which have no section isomorphic with $\mathcal{L}_{1,p}$ for some $p$.

In section \ref{sec3} we realize $\mathcal{L}_{1,p}$ as a group generated by finite automaton of Mealy type with $p$ states over an alphabet on $p$ symbols.  This allows to realize $\mathcal{L}_{1,p}$ as a group acting on a $p$-ry rooted tree in a self-similar manner.  Also this allows to identify a big part of the lattice of $p$-power subgroups with a set of vertex stabilizers for the action of a group on a tree.
Our goal  (not achieved yet) is complete understanding of the lattice of subgroups of $\mathcal{L}_{n,p}$.  Even for subgroups of finite index or normal subgroups this
is not an easy problem. Our first result is the description up to isomorphism of all subgroups of $\mathcal{L}_{n,p}$. In what follows we will skip index $p$, keeping it fixed while $n$ can vary, and denote the group just by $\mathcal{L}_n$, to lighten the notation.  We show in theorem~\ref{3_3} that
either a subgroup of $\mathcal{L}_n$ sits inside the base group $\mathcal{A}_n$ of wreath product and hence is
abelian, or it is isomorphic to $\mathcal{L}_k$ for some $k$.  As a corollary of our arguments we get that a
subgroup of index $p$ of $\mathcal{L}_n$ is   isomorphic  to $\mathcal{L}_n$ or  to $\mathcal{L}_{pn}$ (both cases
occur). This will show in particular that $\mathcal{L}_n$ has a descending chain with trivial intersection of subgroups of primary index
a power of $p$ that consists of groups isomorphic to $\mathcal{L}_n$. % Napevno, tut potribno schos' dodaty, trivial intersection vidrazu ne vyplyvae
 Groups with this property are called scale
invariant. Around 2005 there was a conjecture supported by Benjamini and some other mathematicians
% (zdetsya Sapir na svoii web trumav taku gipotezu, perevirtechu she e) NE ZNAJSHOV
that finitely generated scale invariant groups are virtually nilpotent.
 The conjecture is not correct as we see. That $\mathcal{L}$ is scale invariant was implicitly
established  in \cite{MR1866850} and in details investigated by V. Nekrasevych and  G{\'a}bor Pete in \cite{MR2763782}, where many other scale invariant groups are constructed.

In Theorem~\ref{max}  we describe all maximal subgroups of $\mathcal{L}_n$  (they happen  to be all  of finite
index).

Recall that a subgroup $H$ of a group $G$  is called weakly maximal (or near maximal, see \cite{Riles}) if it has infinite index, but every
subgroup of $G$ containing $H$ as a proper subgroup has finite index in $G$.  All weakly
maximal subgroups of $\mathcal{L}_n$ are described in Theorem~\ref{3_9}, moreover it is determined which of them
are closed in pro-$p$ topology.

A property of $\mathcal{L}_n$ called \emph{hereditary free from finite normal subgroups} is established which, according to results of \cite{Grig}, allows
to conclude that any action of $\mathcal{L}_n$ on the rooted tree induces topologically free action on the
boundary of the tree.  We get a formulae for a number of subgroups of $\mathcal{L}_n$ of index $m$ (Proposition~\ref{3_6}) which allows to obtain an upper bound on subgroup growth  (Proposition~\ref{3_7}). A description of pro-$p$-completion of $\mathcal{L}_n$ is presented in Lemma~\ref{3_19}.

Our main tool to study subgroup structure of  $\mathcal{L}_n$ is based on technique of association to a
group a suitable module on which group acts.  For the case  of solvable groups this technique was
successfully used by P.Hall and many other researchers   \cite{MR0110750}. %  dvi klasuchni statti P.Halla. -Perevirte bud' laska chy ce ta stattia; drugu stattiu ne znajshov
We also follow this line considering the base subgroup as a module over the ring of Laurent polynomial in one variable (in fact different module structures are used). 
This allows us to associate to a subgroup $H\leq \mathcal{L}_n$ a triple $(s,V_0,v)$  consisting  of $s\in\mathbb{N}$, $V_0\leq\mathcal{A}_n$, $v\in\mathcal{A}_n$ that
satisfies certain properties listed in Lemma~\ref{3_1} and Corollary~\ref{3_2}. In  case of a  normal subgroup the triple
satisfies more constrains (Lemmas~\ref{normal_group} and \ref{one_more}). Theorem~\ref{prof} gives a description of the lattice of normal
subgroups of $\mathcal{L}_{1,p}$. As a result we make some contribution to the normal subgroup growth (Theorem~\ref{norm1} and Corollary~\ref{norm2}).
Theorems \ref{closedd} and \ref{closeddp} describe when the subgroup is closed in profinite and pro-$p$ topologies in terms of the associated triples.

The last part of the paper deals with one interesting recent notion of rank gradient of a subgroup. It was introduced by Lackenby and
studied by him, Abert, Jaikin-Zapirain, Luck, Nikolov, and other researchers.  In many cases (in particular under
certain amenability assumptions) it is equal to zero. A basic open question about gradient rank in the situation when a group is amenable is Question~\ref{q2} formulated at section \ref{sec4}.

% One of our results (namely Theorem~\ref{rrr}) shows that the answer is yes in the case of groups $\mathcal{L}_{n,p}$.

 In a situation of zero gradient rank and when a subgroup $H \leq G$ is
presented as the intersection of descending sequence of subgroups $\{H_m\}$ of finite index it is reasonable
to study a relation between the index $[G:H_m]$ and rank $rk(H_m)$ (i.e. a minimal number of generators) as a
function $rg(m)$ of natural argument $m$.  It is interesting to know what type of asymptotic of such
functions arises, how this function, which we call relative gradient rank, depends on presentation of subgroup in the form of intersection etc. Very
little is known about this. The first observation in this direction in the case of lamplighter group was made
in \cite{AG}, Theorem~\ref{4.2} provides further information. % jaki vyhidni danni u cieji statti?

After the case of Lamplighter type groups is treated the next step would be to consider the case of metabelian groups %(starting with Baumslag-Soliter groups) 
and more generally of solvable groups and even of elementary amenable groups \cite{Day}, \cite{CGH}.

Another option is the study of the case of nonelementary amenable groups starting with the $3$-generated $2$-group of intermediate growth constructed by the first author in \cite{MR565099}. 

Acknowledgements.  We would like to express our thanks to Y. de Cornulier and D.Segal. We are also indebted to Anna Erschler and Laboratoire de MathŽmatiques d'Orsay for their hospitality. The major part of the work was done while the first author was visiting it and the second author was a postdoc there.

\section{Preliminary facts}

We consider groups $$\mathcal{L}_n=({\mathbb{Z}/p\bz})^n\wr\mathbb{Z}=\oplus_\mathbb{Z}({\mathbb{Z}/p\bz})^n\rtimes\mathbb{Z}.$$ Let $\mathcal{A}_n$ denote the subgroup $\oplus_\mathbb{Z}({\mathbb{Z}/p\bz})^n$ of $\mathcal{L}_n$. Let $x$ denote the right shift by one on $\mathcal{A}_n$, considered as a generator of active group $\bz$.
%We will often consider elements of $N\rtimes\mathbb{Z}$ that belong to $N$; in that case we will identify $w\in{}N$ and as $(w,0)\in{}N\rtimes\mathbb{Z}$. 
Elements of $\mathcal{L}_n$ are written as $(v,s)$ where $v\in\mathcal{A}_n$ and $s\in\mathbb{Z}$. If $g=(v,s)$ and $h=(w,t)$ are elements of $\mathcal{L}_n$ then the direct computation gives that $gh=(v+x^sw,s+t)$ and $g^{-1}=(-x^{-s}v,-s)$.

\subsection{The ring $R=\mathbb{F}_p[x,x^{-1}]$}

 Let $\bb$ be the simple field of characteristic $p$. In the following we will make repeated use of the properties of the ring $\mathbb{F}_p[x,x^{-1}]$ of Laurent polynomials, which we denote by $R$. Note also that $R$ is just $\mathbb{F}_p[\mathbb{Z}]$, the group ring of $\mathbb{Z}$ over $\mathbb{F}_p$.

\begin{lemma}
$R$ is principal ideal domain. The invertible elements of $R$ are $\{ax^q|a\in\mathbb{F}_p^*, q\in\mathbb{Z}\}$. Each ideal $I$ of $R$ is of the form $Rf$ where $f\in{}\mathbb{F}_p[x]$, $f(0)=1$, such $f$ is uniquely defined by $I$, up to multiplication by a nonzero element of $\mathbb{F}_p$. The factor ring $R/I$ is isomorphic to $\mathbb{F}_p[x]/f\mathbb{F}_p[x]$. Therefore, each element of $R/I$ has the unique representative polynomial $g$ with degree of $g$ less then $f$.
\end{lemma}
\begin{proof}
Let $I$ be an ideal in $R$. Note that $\mathbb{F}_p[x]\subset{}R$. Then $I\cap{}\mathbb{F}_p[x]$ is an ideal in $\mathbb{F}_p[x]$ and since it is principal ideal domain $I\cap{}\mathbb{F}_p[x]=f\mathbb{F}_p$ for some $f\in{}\mathbb{F}_p[x]$. Then $Rf\subset{}I$. For $h\in{}I$ $h=x^{-k}g$ for some $k\in\mathbb{N}$ and $g\in{}\mathbb{F}_p[x]$. Thus $g\in{}I\cap{}\mathbb{F}_p[x]=f\mathbb{F}_p[x]$, and so $h=x^{-k}fa\in{}Rf$. Therefore $Rf=I$. Note that if $f=x^kf_0$ and $f_0(0)=1$ then $Rf=Rf_0$. Thus we have proved that each ideal of $R$ is generated by a single element of the form $f\in{}\mathbb{F}_p[x]$, $f(0)=1$. If $Rf_1=Rf_2$ then $f_1=uf_2$ with $u\in{}R$ invertible in $R$, therefore $u=ax^k$. But since both $f_i(0)=1$ we have that $k=0$.

To prove that $R/I$ is isomorphic to $\mathbb{F}_p[x]/f\mathbb{F}_p[x]$ when $I=Rf$, note that $\mathbb{F}_p[x]$ is a subring of $R$ and $I\cap{}\mathbb{F}_p[x]=f\mathbb{F}_p[x]$, therefore $\mathbb{F}_p[x]/f\mathbb{F}_p[x]$ is subring of $R/I$. It is left to note that since $f(0)=1$ we have that $f=a+xg$ for some $g\in{}\mathbb{F}_p[x]$ and $a\in\mathbb{F}_p^*$ and thus $x$ is invertible in $\mathbb{F}_p[x]/f\mathbb{F}_p[x]$.
\end{proof}

Note that it follows in particular that for $f,f_1\in{}\mathbb{F}_p[x]$ with $f(0)\neq 0$, $f_1$ divides $f$ in $R$ if and only if $f_1$ divides $f$ in $\mathbb{F}_p[x]$. Therefore in the following we will just say that $f_1$ divides $f$ and write $f_1|f$ meaning that it happens both in $R$ and in $\bb[x]$.

\begin{definition}
Let $f\in R$. $\Deg(f)$ is the dimension $\dim_{\mathbb{F}_p}R/fR$ of vector space $R/fR$ over $\mathbb{F}_p$. Note that it is also equal to the $m-k$, where $m$ is the largest dergee of $x$ in $f$ with the  nonzero coefficient and $k$ it the smallest degree of $x$ with nonzero coefficient.
\end{definition}
Note that we use capital 'D' here in order to avoid confusion with standard notion of degree when $f\in \mathbb{F}_p[x]$. In this case $\deg f$ is the dimension of $\bb[x]/ f\bb[x]$, and $\Deg f$ is the dimension of $R/ fR$, clearly $\Deg f\leq\deg f$ and they are equal if and only if $f(0)\neq 0$.  

We will be also interested in the following facts about the ring $M_n(R)$ of $n\times n$ matrices over $R$ and the right module $R^n$ over $M_n(R)$ with the standard action of $M_n(R)$ (but often we also consider $R^n$ as  a left module over $R$ with the action described in each particular case).

The next lemma is basic for many considerations (and true for matrices with coefficients in any principal ideal domain).
\begin{lemma}\label{decomp}
Let $g\in M_n(R)$. Then there exist $a,b\in M_n(R)$ such that $g=adb$ where $d\in M_n(\mathbb{F}_p[x])$ is the diagonal matrix, and $\det(a),\det(b)$ are invertible in $R$ (i.e. $a,b\in GL_n(R)$).
\end{lemma}
\begin{corollary}
 $\dim_{\mathbb{F}_p}R^n/gR^n=\Deg (\det (g))$.
\end{corollary}
\begin{proof}
Indeed, if $g=adb$ is the decomposition then $\det(g)$ equals to a product of $\det(d)$ and an invertible element of $R$, thus $R/\det(g)R=R/\det(d)R$ and therefore $\Deg(\det(g))=\Deg(\det(d))$. Also, $R^n/gR^n=R/adR^n$ is isomorphic to $R^n/dR^n$. Now if $d_1,\dots,d_n$ are diagonal entries of $d$ then $\det(d)=d_1\cdots d_n$ and $R^n/dR^n=\oplus_{j=1}^n R/d_jR$.
\end{proof}

Note that if $g,g_1\in M_n(R)$ then $gR^n\subset g_1R^n$ if and only if $g=g_1h$ for some $h\in M_n(R)$, that is if $g_1$ divides $g$ in $M_n(R)$.

\begin{lemma}\label{maxmatrix}
Let $g\in M_n(R)$. Submodule $gR^n$ is maximal in $R^n$ if and only if $\det(g)$ is nonzero and irreducible in $R$.
\end{lemma}
\begin{proof}
Let $g=adb$ be the decomposition from lemma \ref{decomp}. Then $\det(g)$ equals to a product of $\det(d)$ and an invertible element of $R$ and $R^n/gR^n$ is isomorphic to $R^n/d R^n$, thus $g R^n$ is maximal if and only if $d R^n$ is maximal. So we reduced the lemma to the case of diagonal matrices. Let $d_1,\dots,d_n$ be the entries on the diagonal of $d$. Then $R^n/dR^n=\oplus_{j=1}^n R/d_jR$, and therefore $d R^n$ is maximal \Iff one of the $d_i$ is irreducible and the others are invertible, but this happens exactly if $\det d=d_1\cdots d_n$ is irreducible.
\end{proof}

Suppose that $U$ is a submodule of $R^n$. Then by the structure theorem for the finitely generated modules over principal ideal domain, $R^n/U$ is isomorphic to $\oplus_{i=1}^n R/q_i^{r_i}R$, where $q_i$ are $0$ or irreducible polynomials in $\mathbb{F}_p[x]$ such that $q_i(0)=1$. Up to permutation, $q_i$ and $r_i$ are uniquely defined by $U$.
\begin{definition}\label{dett}
Let $U$ be a submodule of $R^n$, with the quotient $R^n/U$ isomorphic to $\oplus_{i=1}^n R/q_i^{r_i}R$ as above. Define $${\det}^*U=\prod_{i:q_i\neq 0}q_i^{r_i}\in\mathbb{F}_p[x].$$
\end{definition}  
%Note that if $U=gR^n$ then ${\det}^*U$ is the (normalized) product of the nonzero diagonal elements of the matrix $d$ from the Lemma \ref{decomp}.

\subsection{Subgroups of $\mathcal{A}_n$}

Let $U$ be an abelian subgroup of $\mathcal{A}_n$. %Recall that $\mathcal{A}_n$ is isomorphic to additive group of $R^n$ 
\begin{definition}\label{expp}
Define $e(U)$ to be the minimal positive $e$ such that $x^eU=U$, if such $e$ exists, and $+\infty$ otherwise.
\end{definition}  
Note that if $x^sU=U$ for some $s>0$ then $e(U)$ divides $s$.

\begin{lemma}\label{mod_structure}
Define a new structure of left $R$-module on $\mathcal{A}_n$ by the rule $x*v=x^sv$, denote this new $R$-module by $\mathcal{A}^s_n$. Then $\mathcal{A}^s_n$ is isomorphic to $R^{ns}$.
\end{lemma}
\begin{proof}
Indeed, let $e_i$, $1\leq{}i\leq{}n$ be the basis of space $(\mathbb{Z}/p\bz)^n$ at place $0$ in $\mathcal{A}_n$. Then the map
\begin{equation}\label{isomorphism}
R^{ns}\ni (f_{ij})_{1\leq{i}\leq{n},0\leq{j}\leq{s-1}}\mapsto\sum_{i,j}f_{ij}(x^s)x^je_i \in \mathcal{A}_n^s,
\end{equation}
where $f_{ij}\in{}R$, is the isomorphism of $R$ modules.
\end{proof}

\begin{definition}\label{ddet}
Suppose that $e(U)$ is finite for some $U\subset\mathcal{A}_n$. By the previous Lemma we can consider $U$ as a submodule of $R^{ne(U)}$. Define ${\det}^*U$ by the definition \ref{dett}.
\end{definition}

We now show that to compute ${\det}^*U$ it is possible to use other  $s$ such that $x^sU=U$.
\begin{proposition}
Suppose $U\subset\mathcal{A}_n$ and $x^sU=U$,  and let $f(x)={\det}^*U$ where $U$ is considered as a submodule of $R^{ns}$. Then we also have $x^{sp}U=U$ and if $g(x)$ is ${\det}^*U$ where $U$ is considered as submodule of $R^{nsp}$, then $f(x)=g(x)$.
\end{proposition}
\begin{proof}
It follows from the following elementary fact. Suppose $q\in R$ is irreducible, and let $M$ be an $R$-module $R/q^kR$ with the multiplication $x*(r+q^kR)=x^pr+q^kR$ and $k=tp+b$ for $0\leq b<p$. Then $M$ is isomorphic to $(R/q^tR)^{p-b}\oplus(R/q^{t+1}R)^b$. 
\end{proof}

\subsection{Tree of cosets}
To study residually finite groups is the same as to study groups acting faithfully on spherically homogeneous rooted trees (see \cite{MR1765119}, \cite{MR1841755}). There is a bijection between such actions and descending chains of subgroups of finite index with trivial core. The next definition illustrates this. 
\begin{definition}
Let $G$ be a group, and $\{H_i|i\geq 1\}$ be the descending chain of subgroups of $G$, with $H_1=G$. 
Construct the following graph: let the vertices be the right cosets of subgroups $H_i$, that is the set of vertices is $\{gH_i|g\in G, i\geq 1\}$. Let $gH_i$ and $g'H_j$ be connected by an edge iff $i=j+1$, $gH_i\subset g'H_j$ or $j=i+1$, $gH_i\supset g'H_j$. Then the constructed graph is a rooted tree with the root $G=H_1$, which is called the tree of cosets corresponding to the chain $\{H_i\}$.
\end{definition}
The tree of cosets is locally finite \Iff $(G:H_i)<\infty$ for all $i$.
Note that each $g\in G$ defines an automorphism of the tree of cosets by the multiplication from the right. 

The rooted tree is called $p$-regular if the number of edges going out from each vertex which is not a root is equal to $p+1$, and the number of vertices going out from the root is equal to $p$. We denote such a tree by $T_p$.  The group $\Aut T_p$ of all automorphisms of the $p$-regular tree can be identified with $\wr_{1}^{\infty}S_p$, where $S_p$ is the group of all permutations on $\{1,\dots,p\}$. If one considers the action of $\bz/p\bz$ on $\{1,\dots,p\}$ by powers of a standard cycle $(1,\dots,p)$, one can consider $\bz/p\bz$ as a subgroup fo $S_p$. Therefore there is a subgroup of $\Aut T_p$ which is isomorphic to $\wr_1^{\infty}(\bz/p\bz)$. In fact, it is the Sylow $p$-subgroup of $\Aut T_p$ as a profinite group, see \cite{MR1841755}.  In the next Proposition we give the following characterization when the tree of cosets is $T_p$, and when the action of $G$ on it lies in $\wr_1^{\infty}(\bz/p\bz)$.
\begin{proposition}\label{tree}
The tree of cosets is $p$-regular if and only if the chain is a $p$-chain, that is $[H_i:H_{i+1}]=p$ for all $i\geq 1$. If the tree of cosets is $p$-regular then automorphisms of the tree defined by elements of $G$ lie in $\wr_1^{\infty}(\mathbb{Z}/p\bz)$ if and only if the chain is subnormal, that is $H_{i+1}$ is normal in $H_i$ for each $i\geq 1$.
\end{proposition}
\begin{proof}
The first statement is obvious. For the second statement note that the fact that automorphisms of the tree defined by elements of $G$ lie in $\wr_1^{\infty}(\mathbb{Z}/p\bz)$ is equivalent to the fact that if for some $a,g\in G$ and $i\geq 1$ we have that $agH_i=gH_i$ then $a$ acts on cosets of $H_{i+1}$ that lie in $gH_i$ trivially or as some fixed full cycle. Note that since $g^{-1}ag\in H_i$ this is equivalent to the fact that for each $i\geq 1$ the action of $H_i$ on $H_i/H_{i+1}$ by right multiplication acts on $H_i/H_{i+1}$ as powers of some fixed full cycle. Since $p$ is prime this is equivalent to $H_{i+1}$ being normal in $H_i$. 
\end{proof}
Note that if $\{H_i\}$ is a descending subnormal $p$-chain then it follows that the stabilizers of the levels of the coset tree are of index a power of $p$ in $G$.

\subsection{Profinite and pro-$p$ topology}
For convenince of the reader we remind some well known definitions. Let $G$ be a group. 
\begin{definition}
$\hat{G}$, which is called the profinite closure of $G$, is the inverse limit of the inverse system $\{G/N\}_N$ where  $N$ runs through all normal subgroups of $G$ with finite index, considered with the projective topology, which is called profinite topology. We denote by $\pi$ the natural homomorphism $\pi:G\rightarrow\hat{G}^{(p)}$. 
\end{definition}
Note that $\pi$ is injective if and only if $G$ is residually finite.
If in this definition we substitute finite index with index a power of $p$, we obtain a definition of pro-$p$ closure.
\begin{definition}
$\hat{G}^{(p)}$, called the pro-$p$ closure of $G$, is the inverse limit of the inverse system $\{G/N\}_N$ where  $N$ runs through all normal subgroups of $G$ with index a power of $p$, considered with the projective topology, which is called pro-$p$ topology. We denote by $\pi_p$ the natural homomorphism $\pi_p:G\rightarrow\hat{G}^{(p)}$.
\end{definition}
Note $\pi_p$ is injective if and only if $G$ is residually $p$-finite (i.e. approximated by finite $p$-groups).
%The above topologies are Hausdorff \Iff the group is residually finite or residually $p$-finite (i.e. approximated by finite $p$-groups) respectively.
\begin{definition}
Let $H$ be a subgroup of $G$. We say that $H$ is closed in profinite (respectively pro-$p$) topology if $\pi(H)$ ($\pi_p(H)$)  is equal to the intersection of the closure of $\pi(H)$ with $\hat{G}$ (respectively $\pi_p(H)$ with $\hat{G}^{(p)}$). 
\end{definition} 
We have the following characterization:
\begin{proposition}\label{pro2}
Suppose $G$ is finitely generated. A subgroup $H$ of $G$ is closed in pro-$p$ topology if and only if there is a descending subnormal $p$-chain of subgroups $H_i$ such that $H_1=G$, $[H_i:H_{i+1}]=p$ for all $i\geq 1$, and $H=\cap_i H_i$.
\end{proposition}
\begin{proof}
Since $G$ is finitely generated we may choose a decreasing sequence $N_j$ of normal subgroups of index a power of $p$ such that $\hat{G}^{(p)}=\varprojlim_j G/N_j$. Then $H$ is closed in pro-$p$ topology iff $H=\cap_j HN_j$, and since $HN_j/N_j\subset HN_{j-1}/N_j$ which is a finite $p$-group, we can construct a descending subnormal $p$-chain between $HN_{j}$ and $HN_{j-1}$ for all $j$.

Now suppose $H=\cap_i H_i$, $G=H_1$ and $[H_i:H_{i+1}]=p$ and $H_{i+1}$ is normal in $H_i$ for $i\geq 1$. Construct the coset tree which corresponds to it. $G$ acts on it by multiplication from the right, and $H$ is the stabilizer of the ray $\{H_i\}_i$. Let $M_i$ be the stabilizer of $i$th level of the tree. Then $M_i$ are normal and of index a power of $p$ in $G$ by the proposition. Moreover, $H=\cap_i HM_i$, which implies that $H=\cap_j HN_j$ (since $N_j$ define the pro-$p$ topology, thus for each $i$ there is $j$ such that $M_i\supset N_j$).
\end{proof}
\begin{remark}
Note that a $2$-chain is always subnormal. This is the only case when in this article even prime play a special role.
\end{remark}

\section{Structure of subgroups}\label{sec3}
Now we again identify $\mathcal{A}_n$ with the additive group of the ring $R^n$ and convert it into the left $R$-module with the action $x:v\mapsto xv$. Geometrically one can view this action as given by a left shift, identifying Laurent polynomials with the bilateral sequences over alphabet $\{0,1,\dots,p-1\}$ (containing only finitely many non-zero entries).
 
Let us introduce the notation $\varphi_t(x)=1+x+\dots+x^{t-1}$. Throughout the article we will use
\begin{lemma}\label{3_1}
\label{group}
Let $V$ be a subgroup of $\mathcal{L}_n$. Then it defines the triple $(s,V_0,v)$, where $s\in\mathbb{N}$ is such that $s\mathbb{Z}$ is the image of projection of $V$ on $\mathbb{Z}$, $V_0=V\cap\mathcal{A}_n$, satisfying $x^sV_0=V_0$, and $v\in\mathcal{A}_n$ is such that $(v,s)\in{}V$. The $v$ is uniquely defined up to addition of elements from $V_0$. For $s=0$ one can choose $v=0$.

Conversely any triple $(s,V_0,v)$ with such properties gives rise to a subgroup of $\mathcal{L}_n$. Two triples $(s,V_0,v)$ and $(s',V'_0,v')$ define the same subgroup if and only if $s=s'$, $V_0=V'_0$ and $v+V_0=v'+V_0$.

Moreover, $V\subset{}V'$ if and only if $s'|s$, $V_0\subset{}V_0'$ and $v=\varphi_{s/s'}(x^{s'})v'\mod{}V_0'$.
\end{lemma}
\begin{proof}
To show the first part of the Lemma we need only to prove that $x^sV_0=V_0$. Indeed, 
%note that $(v,s)\in{}V$ by construction $v$, and it is easy to compute that the inverse of it is $(\sigma^{-s}v,-s)$, which is also in $V$. Hence if $w\in{}V_0$ we have
\[
(v,s)(w,0)(-x^{-s}v,-s)=(x^sw,0)\in{}V,
\]
and so $x^sw\in{}V\cap\mathcal{A}_n=V_0$. Thus $x^sV_0\subset{}V_0$. The converse inclusion is proved in a similar way.

Now suppose we are given a triple $(s,V_0,v)$ such that $V_0$ is a subgroup of $\mathcal{A}_n$ and $x^sV_0=V_0$. Let $V$ be the subgroup of $\mathcal{L}_n$ generated by $V_0$ and $(v,s)$. To show that the triple $(s,V_0,v)$ is a triple for $V$ we need to show that the projection of $V$ onto $\mathbb{Z}$ is $s\mathbb{Z}$ and that $V\cap\mathcal{A}_n=V_0$. Since $V_0\subset\mathcal{A}_n$ the projection of $V$ on $\mathbb{Z}$ is indeed $s\mathbb{Z}$. Clearly, $V_0\subset{}V\cap\mathcal{A}_n$. To prove the converse inclusion note first that if $s=0$ then $v\in{}V_0$ and therefore $V=(V_0,0)$. Suppose that $s\geq{}1$. Then by the equality above any element of $V$ is of the form $(w,0)(v,s)^k$ for some $w\in{}V_0$ and $k\in\mathbb{Z}$, and $(w,0)(v,s)^k\in{}V_0$ if and only if $k=0$, which implies that $V_0=V\cap\mathcal{A}_n$.

Suppose now that we have two triples $(s,V_0,v)$ and $(s',V'_0,v')$, and $V$ is the subgroup generated by $V_0$ and $(v,s)$, while $V'$ is the subgroup generated by $V'_0$ and $(v',s')$. It is left to prove that $s=s'$, $V_0=V'_0$, and $v+V_0=v'+V_0$. First two equalities immediately follow from the definitions of $s$ and $V_0$ in terms of $V$. On the other hand we know that $(v,s)$ and $(v',s)$ both belong to $V=V'$. It follows that
\[
(v',s)(v,s)^{-1}=(v',s)(-x^{-s}v,-s)=(v'-v,0)
\]
also belongs to $V$. Thus $v'-v\in{}V_0$ and therefore $v+V_0=v'+V'_0$.

To prove the last part, note that both $(v,s)$ and $(\varphi_{s/s'}(x^{s'})v',s)=(v',s')^{s/s'}$ belong to $V'$, and therefore $v=\varphi_{s/s'}(x^{s'})v'\mod{}V_0'$. Conversely, if $s|s'$, $V_0\subset{}V_0'$ and  $v=\varphi_{s/s'}(x^{s'})v'+w$ for some $w\in{}V_0'$, then $(v,s)=(w,0)(v',s')^{s/s'}\in{}V'$, and therefore $V\subset{}V'$.
\end{proof}

\begin{remark}
Subgroup $V$ belongs to $\mathcal{A}_n$ if and only if $s=0$ in the triple corresponding to $V$.
\end{remark}

We have the description of normal subgroups of $\mathcal{L}_n$ in terms of the triple.
\begin{lemma}
\label{normal_group}
Let $V$ be a subgroup of $\mathcal{L}_n$. Then $V$ is normal if and only if the corresponding triple $(s,V_0,v)$, defined in Lemma~\ref{group} satisfies the additional properties that $x{}V_0=V_0$, $(1-x^s)\mathcal{A}_n\subset{}V_0$, and $(1-x)v\in{}V_0$.
\end{lemma}

\begin{proof}
 Take $w\in{}V_0$. Then, both
 \[
 (0,1)(w,0)(0,-1)=(x{w},0),
 \]
  \[(0,-1)(w,0)(0,1)=(x^{-1}w,0),
  \] belong to $V$ since $V$ is normal. Thus both $x{w}$ and $x^{-1}w$ also belong to $V_0$. It follows that $x{}V_0=V_0$.

Take now $w\in\mathcal{A}_n$. Then $(w,0)(v,s)(-w,0)=((1-x^s)w+v,s)\in{}V$, since $(v,s)\in{}V$ and $V$ is normal. It follows that $(1-x^s)w+v+V_0=v+V_0$, thus $(1-x^s)w\in{}V_0$, and it is true for any $w\in\mathcal{A}_n$. In a similar way $(0,1)(v,s)(0,-1)=(x{v},s)\in{}V$, therefore $v+V_0=x{v}+V_0$, and thus $(1-x)v\in{}V_0$.
The proof of the converse statement, that if $(s,V_0,v)$ satisfies such additional properties then $V$ is normal, is also straightforward.
\end{proof}

\begin{corollary}\label{one_more_cor}
Let $V$ is a normal subgroup of $\mathcal{L}_n$ with the triple $(s,V_0,v)$. Then if $s>0$ it has finite index in $\mathcal{L}_n$. If $n=1$ and $V$ is nontrivial, $V_0$ has finite index in $\mathcal{A}_n$.
\end{corollary}
\begin{proof}
If $s>0$ then $(x^s-1)\mathcal{A}_n$ has finite index in $\mathcal{A}_n$, therefore since $(x^s-1)\mathcal{A}_n\subset V_0$, we obtain that $V_0$ also has finite index in $\mathcal{A}_n$ and hence $V$ has finite index in $\mathcal{L}_n$.

Let $n=1$. Then if $s>0$ we are done, so suppose $s=0$ and $V=V_0\subset\mathcal{A}_1$. By the previous lemma $xV_0=V_0$, hence identifying $\mathcal{A}_1$ with $R$ we obtain that $V_0$ is the nontrivial ideal of $R$. Any such ideal is of the form $fR$ for some $f\in\mathbb{F}_p[x]$ and hence is of finite index in $R$.
\end{proof}

\begin{corollary}\label{3_2}
Let $V$, $W$ be the subgroups of $\mathcal{L}_n$, with corresponding triples $(s,V_0,v)$ and $(t,W_0,w)$. Let $U=V\cap{}W$. Choose $r>0$ minimal with the property that both $s$ and $t$ divide $r$ and moreover
\[
\frac{1-x^r}{1-x^s}v=\frac{1-x^r}{1-x^t}w\mod{}(V_0+W_0),
\]
and if no such number exists put $r=0$.

 If $r=0$ then $V\cap W=V_0\cap W_0\subset\mathcal{A}_n$, in particular $V\cap W=\{0\}$ if $r=0$ and $V_0\cap W_0=\{0\}$. 

If $r>0$ then exist $v_0\in{}V_0$ and $w_0\in{}W_0$ such that 
\[
\frac{1-x^r}{1-x^s}v+v_0=\frac{1-x^r}{1-x^t}w+w_0.
\]
Denote this element by $u$. Then $(r,V_0\cap{}W_0,u)$ is the triple that corresponds to the subgroup $U$. 
\end{corollary}
\begin{proof} Observe that $u=\varphi_{r/s}(x^s)v+v_0$ and that $(1-x^r)/(1-x^s)=\varphi_{r/s}(x^s)$.
Assume $r=0$. Suppose there is $g=(x,q)\in V\cap W$ with $q\neq 0$. It follows that  $(x,q)=(v_0,0)(v,s)^{n}$ for some $v_0\in V_0$ and integer $n$. But $(v,s)^n=(\varphi_n(x^s)v,sn)$ and therefore $x=v_0+\varphi_n(x^s)v$ and $q=sn$. Analogously since $g=(x,q)\in W$ it follows that  $q=tm$ and $x=w_0+\varphi_m(x^t)w$ for some  $w_0\in W_0$ and some integer $m$. Therefore $\varphi_{q/s}(x^s)v=\varphi_{q/t}(x^t)w\mod(V_0+W_0)$ and thus $r>0$, contradiction. It follows that for any $g\in V\cap W$ the projection onto $\bz$ is $0$, thus $V\cap W\subset \mathcal{A}_n$.

Let now $r>0$. The subgroup defined by the triple $(r,V_0\cap{}W_0,u)$ belongs both to $V$ and $W$. Indeed, we just need to prove that $u=\varphi_{r/s}(x^s)v\mod{}V_0$ and $u=\varphi_{r/t}(x^t)w\mod{}W_0$, but this directly follows from the definition of $r$ and $u$.

Conversely, suppose that a subgroup $U'$ with the triple $(r',U'_0,u')$ belongs to both $V$ and $W$. We need to prove that $U'\subset{}U$, that is, by the Lemma, that $U'_0\subset V_0\cap{}W_0$, $r|r'$ and $u'=\varphi_{r'/r}(x^r)u\mod{}V_0\cap{}W_0$.

Since $u'=\varphi_{r'/s}(x^s)v\mod{}V_0$ and $u'=\varphi_{r'/t}(x^t)v\mod{}W_0$ it follows that $\varphi_{r'/s}(x^s)v=\varphi_{r'/t}(x^t)w\mod{}V_0+W_0$, therefore by the minimality of $r$ $r\leq{}r'$. Let $r'=ar+b$ with $b<r$. Then
\[
\varphi_{r'/s}(x^s)v=\frac{1-x^{r'}}{1-x^s}v=\frac{1-x^b}{1-x^s}v+x^b\frac{1-x^{ar}}{1-x^r}\varphi_{r/s}(x^s)v,
\]
\[
\varphi_{r'/t}(x^t)w=\frac{1-x^{r'}}{1-x^t}w=\frac{1-x^b}{1-x^t}w+x^b\frac{1-x^{ar}}{1-x^r}\varphi_{r/t}(x^t)w,
\]
and thus 
\[
\frac{1-x^b}{1-x^s}v=\frac{1-x^b}{1-x^t}w\mod V_0+W_0.
\]
Since $b<r$ we obtain that $b=0$ and therefore $r$ divides $r'$.

To show that $u'=\varphi_{r'/r}(x^r)u\mod{}V_0\cap{}W_0$ note that 
\[
\varphi_{r'/r}(x^r)u=\frac{1-x^{r'}}{1-x^r}u=\frac{1-x^{r'}}{1-x^s}v+\varphi_{r'/r}(x^r)v_0=\frac{1-x^{r'}}{1-x^t}w+\varphi_{r'/r}(x^r)w_0,
\]
and since $r$ is divisible by both $s$ and $t$, $\varphi_{r'/r}(x^r)v_0\in{}V_0$ and $\varphi_{r'/r}(x^r)w_0\in{}W_0$. Therefore $u'=\varphi_{r'/r}(x^r)u\mod{}V_0$ and $u'=\varphi_{r'/r}(x^r)u\mod{}W_0$, which finishes the proof.
\end{proof}

 The structure of subgroups of $\mathcal{L}_n$ can be described as follows.
\begin{theorem}\label{3_3}
Suppose $V$ is a subgroup of $\mathcal{L}_n$. Then either $V$ is a subgroup of $\mathcal{A}_n$ or $V$ is isomorphic to $\mathcal{L}_k$ for some $k\in\mathbb{N}$. In particular if $V$ is of finite index in $\mathcal{L}_n$ then $V$ is isomorphic to $\mathcal{L}_{ns}$ where $s$ is the projection of $V$ onto $\mathbb{Z}$.% as described in the previous lemma. 
\end{theorem}
\begin{proof}
Let $(s,V_0,v)$ be some triple corresponding to $V$. If $s=0$ then $V=V_0$ is a subgroup of $\mathcal{A}_n$. Suppose that $s\geq{}1$. Recall (Lemma \ref{mod_structure}) that we can consider $\mathcal{A}_n$ as a left $R$-module, which we denoted $\mathcal{A}_n^s$, by defining the action of $x$ on $v$ as $x*v=x^sv$, $v\in\mathcal{A}_n$. Then $V_0$ is $R$ submodule. By Lemma \ref{mod_structure} $\mathcal{A}_n^s$ is a free $R$ module of rank $ns$.  Since $R$ is a principal ideal domain, it is noetherian, and thus $V_0$ is finitely generated as $R$ module, and therefore, as a submodule of the free finitely generated $R$ module $R^{ns}$ $V_0$ is also free, i.e. $V_0$ is isomorphic to $R^k$ for some $k$. $V_0$ is of finite index in $\mathcal{A}_n$ iff $k=ns$.
Now let $H$ be the subgroup of $V$ generated by $(v,s)$. Then $V_0$ is normal in $V$, $V_0\cap{}H=\{0\}$, $H$ is isomorphic to $\mathbb{Z}$. It follows that $V$  is isomorphic to $V_0\rtimes\mathbb{Z}$, where the action of the generator of $\mathbb{Z}$ on $V$ is given by the map $x^s$. It follows, by above, that $V$ is isomorphic to $R^k\rtimes\mathbb{Z}$ where $\mathbb{Z}$ acts on $R^k$ by multiplication by $x$. Therefore, $V$ is isomorphic to $\mathcal{L}_k$. If $V$ is of finite index in $\mathcal{L}_n$ then $k=ns$.
\end{proof}

\begin{corollary}\label{3_4}
Any finitely generated subgroup of $\mathcal{L}_n$ is closed in the profinite topology.
\end{corollary}
This result also follows from Proposition 3.19 in \cite{Corn}. Groups with stated property are called LERF or subgroup separable.

\begin{proof}
Let $V$ be a finitely generated subgroup of $\mathcal{L}_n$ and let $(s,V_0,v)$ be the corresponding triple. If $s=0$ then $V\subset\mathcal{A}_n$, and since it is finitely generated, it is finite, hence closed.
If $s>0$, then $V$ belongs to the subgroup $W$ with the triple $(s,\mathcal{A}_n,0)$. $W$ has finite index in $\mathcal{L}_n$, and therefore $V$ is profinitely closed in $\mathcal{L}_n$ if and only if it is closed in $W$. Moreover, $W$  is isomorphic to $\mathcal{L}_{ns}$, and $V$, considered as a subgroup of $\mathcal{L}_{ns}$, has the projection on $\mathbb{Z}$ equal to $\mathbb{Z}$. Hence we have reduced to the case when $s=1$.
For any $v\in\mathcal{A}_n$, there is an automorphism $\varphi$ of $\mathcal{L}_n$ which is identical on $\mathcal{A}_n$ and $\varphi(0,1)=(v,1)$. Since the property of being profinitely closed is preserved under an automorphism, and the triple for $\varphi(V)$ is $(1,V_0,0)$, we have reduced to the case when $s=1$ and $v=0$. Changing the basis in $\mathcal{A}_n$ we reduce to the case when $V_0=\oplus_{i=1}^{n}f_iR$, which is closed in the profinite topology.
\end{proof}

\begin{corollary}
Any subgroup of index $p$ of $\mathcal{L}_n$ is isomorphic either to $\mathcal{L}_n$ or to $\mathcal{L}_{pn}$.
\end{corollary}
\begin{proof}
Indeed, if $V$ is of index $p$ then the corresponding $s$ defined in Lemma~\ref{group} is either $1$ or $p$. By the previous Lemma, if $s=1$ then $V$ is isomorphic to $\mathcal{L}_n$, and if  $s=p$ then $V$ is isomorphic to $\mathcal{L}_{pn}$.
\end{proof}

\subsection{Subgroup Growth}
Now we compute the number of subgroups of given index of $\mathcal{L}_n$. First let us compute the number of submodules of $R^k$ of given finite codimension.
Define
\[
b_t(R^k)=|\{M\subset{R^k}:M\text{ is $R$ submodule and }\dim_{\mathbb{F}_p}R^k/M=t\}|.
\]
\begin{lemma}\label{submodulegrowth}
$b_t(R^k)=p^{tk}-p^{(t-1)k}$ for $t>0$, $b_0(R^k)=1$.
\end{lemma}
\begin{proof}
We use a similar method that was used in the case of $\mathbb{Z}^k$ for counting subgroup growth (see for instance \cite{LS}). It is based on consideration of endomorphisms of $R^k$. Since $R$ is principal ideal domain, every submodule of $R^k$ is isomorphic to $R^s$, $s\leq{}k$, and if it is of finite index, then $s=k$. Therefore, it is equal to $gR^k$, for some $g\in{}M_k(R)$, $\det(g)\neq{}0$. Moreover, the index of $gR^k$ in $R^k$ is equal to the number of elements in $R/(\det{g})$. Also, $g_1R^k=g_2R^k$ if and only if $g_2^{-1}g_1\in{}GL_k(R)$. It is easy to see that the orbit representatives of the action by $GL_k(R)$ by the multiplication from the right on the set $\{g\in{}M_k(R)|\det(g)\neq{}0\}$ are lower triangular matrices $(g_{ij})$ such that $g_{ij}\in{}\mathbb{F}_p[x]$ for all $i,j$, $g_{ii}(0)=1$ for all $i$, $g_{ij}=0$ for $i<j$, and the degree of $g_{ij}$ is less than $g_{jj}$ for all $i,j$. 

Introduce function $\phi$ by $\phi(s)=(p-1)p^{s-1}$ if $s\geq 1$ and $\phi(0)=1$. Note that $\sum_{s=1}^n\phi(s)=p^n$ and that the number of polynomials $g\in{}F_p[x]$ such that $g(0)=1$ and the degree of $g$ is $s$ is equal to $\phi(s)$. Then the number of all submodules of codimension $t$ (that is, of index $p^t$) of $R^k$ is equal to the number of the lower triangular matrices described above with sum of degrees $\sum_i\deg(g_{ii})=t$, that is to
\[
\sum_{x_1+\dots+x_k=t}\phi(x_1)\dots\phi(x_k)p^{x_2+\dots+(k-1)x_k}.
\] 
To compute it we first compute 
\begin{equation}
\begin{aligned}
&\sum_{x_1+\ldots+x_k\leq{}t}\phi(x_1)\cdots\phi(x_k)p^{x_2+\dots+(k-1)x_k}=\\&\sum_{x_2+\dots+x_k\leq{}t}\phi(x_2)\cdots\phi(x_k)p^{x_2+\dots+(k-1)x_k}\sum_{x_1=0}^{t-(x_2+\dots+x_k)}{\phi(x_1)}=\\& p^t\sum_{x_2+\dots+x_k\leq{}t}\phi(x_2)\cdots\phi(x_k)p^{x_3+\dots+(k-2)x_k}=\dots=p^{kt},
\end{aligned}
\end{equation}
therefore $b_t(R^k)=p^{tk}-p^{(t-1)k}$
   
\end{proof}
\begin{proposition}\label{3_6}
Let $a_m(\mathcal{L}_n)$ be the number of subgroups of $\mathcal{L}_n$ of index $m$. Then
\[
a_m(\mathcal{L}_n)=\sum_{t:p^t|m}b_t(R^{{\frac{nm}{p^t}}})p^t=\sum_{t\geq{1}:p^t|m}(p^{\frac{nmt}{p^t}}-p^{\frac{nm(t-1)}{p^t}})p^t+1.
\]
\end{proposition}
\begin{proof}
If $V$ is the subgroup of $\mathcal{L}_n$ and $(s,V_0,v)$ is the corresponding triple, then its index is equal to $s[\mathcal{A}_n:V_0]$. The number of subgroups with fixed $s$ and $V_0$ is equal to number of coset representatives $v$, i.e. to the index $[\mathcal{A}_n:V_0]$. The formula now follows.
\end{proof}

\begin{proposition}\label{3_7}
Let $s_m(\mathcal{L}_n)=\sum_{1\leq{}d\leq{}m}a_d(\mathcal{L}_n)$. Then 
\[
s_m(\mathcal{L}_n)\geq{}p(p^{n[m/p]+1}-1),
\]
where $[m/p]$ is the integer part of $m/p$.
\end{proposition}
\begin{proof}
First we write $s_m(\mathcal{L}_n)$ as
\begin{equation*}
\begin{aligned}
&s_m(\mathcal{L}_n)=\\&m+\sum_{t,k\geq{1}:p^tk\leq{m}}(p^{nkt}-p^{nk(t-1)})p^t=m+\sum_{t=1}^{[\log_pm]}p^t\sum_{k=1}^{[m/p^t]}(p^{nkt}-p^{nk(t-1)})=\\&
m-p\left[\frac{m}{p}\right]+\frac{p^{n+1}}{p^n-1}\left[p^{n[m/p]}-1\right]-\frac{p^{n+2}}{p^n-1}\left[p^{n[m/p^2]}-1\right]+\\& \frac{p^{2n+2}}{p^{2n}-1}\left[p^{2n[m/p^2]}-1\right]+
\sum_{t=3}^{[\log_pm]}p^t\sum_{k=1}^{[m/p^t]}(p^{nkt}-p^{nk(t-1)}),
\end{aligned}
\end{equation*}
where
\[
\sum_{t=3}^{[\log_pm]}p^t\sum_{k=1}^{[m/p^t]}(p^{nkt}-p^{nk(t-1)})\leq \frac{p^{3n}}{p^{3n}-1} p^{3(nm/p^3+1)}\log_pm
\]
\end{proof}
\begin{remark}
Note that for $n=1$ this bound gives a slight improvement over the bound in proposition 3.2.1 in \cite{LS}, where it was shown that $s_m(\mathcal{L}_1)>c^m$ for $c<p^{1/p^2}$ and big enough $m$. 
\end{remark}
Now we are going to consider the structure of maximal and weakly maximal subgroups of $\mathcal{L}_n$ (where by weakly maximal we understand those subgroups of infinite index such that any larger subgroup is of finite index).
\begin{theorem}\label{max}
Each maximal subgroup has finite index and is either of the form $\mathcal{A}_n\rtimes{q}\mathbb{Z}$ for some prime $q\in\mathbb{Z}$, in which case its index is $q$, or its corresponding triple is $(1,V_0,v)$ such that ${\det}^*V_0$ is irreducible polynomial, in which case its index is $p^{\deg({\det}^*V_0)}$. Maximal subgroups of the form $\mathcal{A}_n\rtimes q\bz$ are normal. Maximal subgroup with the triple $(1,V_0,v)$ is normal \Iff ${\det}^*V_0=1-x$.

\end{theorem}
\begin{proof}
Suppose $V$ is a maximal subgroup and $(s,V_0,v)$ is the corresponding triple. Then there are three possible cases. If $s=0$ then $V=V_0\subset\mathcal{A}_n$, and since it is maximal, it must be equal to $\mathcal{A}_n$. However, $\mathcal{A}_n$ is not maximal, as it is contained in, say, $\mathcal{A}_n\rtimes{2}\mathbb{Z}$.

If $s>1$ then, by maximality $V_0=\mathcal{A}_n$, thus $v$ can be taken to be $0$, therefore the group is $\mathcal{A}_n\rtimes{s}\mathbb{Z}$. If $q$ is any prime divisor of $s$ then $\mathcal{A}_n\rtimes{s}\mathbb{Z}$ is contained in $\mathcal{A}_n\rtimes{q}\mathbb{Z}$, therefore for maximal subgroups $s$ is prime. On the other hand if a subgroup with triple $(q,\mathcal{A}_n,0)$ is contained in the nontrivial subgroup with triple $(t,W,w)$ then clearly $W=\mathcal{A}_n$ and $s$ divides $q$. Since $q$ is prime, $s=q$ or $s=1$. By nontriviality $s=q$ is the only choice. Thus $\mathcal{A}_n\rtimes{q}\mathbb{Z}$ is maximal for any prime $q$. It is clear that its index is equal to $(\mathbb{Z}:q\mathbb{Z})=q$.

Finally let $s=1$ and $(1,V_0,v)$ be the corresponding triple for a maximal subgroup $V$. By the lemma above $x{}V_0=V_0$. By considering $\mathcal{A}_n$ as a module over $R$ isomorphic to $R^n$ we obtain that $V_0$ is its submodule, and therefore isomorphic to $gR^n$ for some $g\in{}M_n(R)$. Moreover $V_0$ is maximal submodule, by the maximality of the subgroup $V$. By Lemma \ref{maxmatrix}, it happens \Iff $\det g$ is irreducible. Notice then that ${\det}^*V_0=u\det(g)$ for some invertibel element $u\in R$. In this case the index of $V$ in $\mathcal{L}_n$ is equal to the index of $V_0$ in $\mathcal{A}_n$, which is equal to $p^{\Deg(\det g)}=p^{\deg({\det}^*V_0)}$ by the corollary after Lemma \ref{decomp}.

It is obvious that $\mathcal{A}_n\rtimes q\bz$ are normal. To consider when a subgroup with the triple $(1,V_0,v)$ such that $V_0$ is isomorphic to $gR^n$ is normal, let $g=adb$ be the product from Lemma \ref{decomp}. Then ${\det}^*V_0=d_1\cdots d_n$, and it is irreducible iff for some $i$ $d_i$ is irreducible and all other $d_j=1$. The group with a triple $(1,gR^n,v)$ is normal iff $(1-x)R^n\subset gR^n$, since $(1-x)v\in gR^n$ follows from it. But since $gR^n=adbR^n=adR^n$, the condition $(1-x)R^n\subset gR^n$ is equivalent to the condition $(1-x)R^n\subset dR^n$, that is $(1-x)R\subset d_k R$ for all $k$, hence each $d_k$ must divide $x-1$. Therefore the maximal subgroup with the triple $(1,gR^n,v)$ is normal iff $\det(g)=u(x-1)$ for some invertible element $u$ of $R$. 
\end{proof}
%Notice that it follows that $\mathcal{L}_n$ is not primitive.

\begin{theorem}\label{3_9}
Let $V$ be a subgroup of $\mathcal{L}_n$, and $(s,V_0,v)$ be its triple. Then $V$ is weakly maximal if and only if either $s=0$ and $V_0=\mathcal{A}_n$, or $s>0$ and the factormodule $\mathcal{A}_n/V_0$ is isomorphic to $R$ as $R$ modules under the identification of $\mathcal{A}_n$ and $R^{ns}$ from Lemma \ref{mod_structure}.  Such weakly maximal subgroup $V$ is closed in pro-$p$ topology \Iff $s=0$ or $s$ is a power of $p$.
\end{theorem}
\begin{proof}
First let us prove that $\mathcal{A}_n$ is weakly maximal. Indeed, it is of infinite index and any larger subgroup must have nontrivial projection onto $\mathbb{Z}$, i.e. its corresponding triple is $(s, \mathcal{A}_n,0)$, with $s\geq{}1$ and hence has finite index $s$. 

Suppose $V$ is a weakly maximal subgroup and let $(s,V_0,v)$ be the corresponding triple. If $s=0$ then $V=V_0\subset\mathcal{A}_n$, and since $V$ is weakly maximal we have equality $V=\mathcal{A}_n$. If $s\geq{}1$ then $\mathcal{A}_n^s/V_0$ is isomorphic as $R$ module to $R/(g_1)\oplus\dots\oplus{}R/(g_k)\oplus{}R^t$, and by weak maximality of $V$ it follows that $k=0$ and $t=1$. Conversely, suppose that $V$ is such a subgroup that $\mathcal{A}_n^s/V_0$ is isomorphic to $R$ as modules over $R$. Suppose it is not weakly maximal, i.e. there is a subgroup $V'$ of infinite index that contains $V$. Let $(s',V'_0,v')$ be the corresponding triple for $V'$. Then $s'$ divides $s$ and $V_0\subset{}V'_0$. Since $s$ is a multiple of $s'$ we have that $x^sV'_0=V'_0$, thus $V'_0$ is an $R$ submodule of $\mathcal{A}_n^s$ that contains $V_0$. Then either $V'_0$ is equal to $V_0$ or $\mathcal{A}_n^s/V'_0=R/(g')$ for some nonzero $g'$. But in the latter case $V'_0$ is of finite index in $\mathcal{A}_n$ and therefore $V'$ is of finite index in $\mathcal{L}_n$, which contradicts our assumption. Thus $V'_0=V_0$. It is left to prove that $s=s'$. Let $s=s'k$. Let $M$ be the factor $\mathcal{A}_n/V_0$. Since $x^{s'}V_0=V_0$, we may consider $M$ as a module over $R$, with multiplication by $x$ given by the action of $x^{s'}$. To show that $k=1$ and therefore $s=s'$ it is left to prove the following lemma:

\begin{lemma}
Let $M$ be the module over $R$, and let $\pi:R\rightarrow{}R$ be the ring homomorphism  given by $\pi(x)=x^k$. Consider new $R$ module $M_\pi$ with the same underlying space as $M$ and with the action of $R$ given by $r*m=\pi(r)m$. Suppose that $M_\pi$ is isomorphic to $R$ as modules over $R$. Then $k=1$.
\end{lemma} 
\begin{proof}
Since $M_\pi$ is isomorphic to $R$ as modules over $R$ there is $m\in{}M$ such that the map, $r\mapsto\pi(r)m$ is bijective. Consider now map $\eta:R\rightarrow{}M$, $r\mapsto{}rm$, and let $R'=\pi(R)\subset{}R$. Then $\eta$ is linear over $\mathbb{F}_p$ and its restriction to $R'$ is bijective. It follows that $R=R'$ and thus $k=1$.
\end{proof}
Note that we have actually shown that if $V$ is weakly maximal and the corresponding $s$ is nonzero, then it is the smallest positive integer such that $V_0$ is invariant under multiplication by $x^s$.

Now let $V$ be a weakly maximal subgroup which is closed in pro-$p$ topology, and suppose $s>0$. By the Proposition \ref{pro2} there is a descending subnormal $p$-chain $\{H_i\}$ such that $V=\cap_i H_i$, and let $(s_i,V_i,v_i)$ be a triple associated to $H_i$. Then $s_i|s_{i+1}|s$ for all $i$, hence the sequence $s_i$ must stabilize. Thus $s_i=p^t$ for some $t$ and all big enough $i$. Since $V=\cap_i H_i$, we obtain $V_0=V\cap\mathcal{A}_n=\cap_i V_i$. The fact that $x^{p^t}V_i=V_i$ for all $i$ implies that $x^{p^t}V_0=V_0$, and since $s$ is smallest such integer we must have $s=p^t$.

Conversely, suppose that $V$ is weakly maximal and $s=p^t$, then $V\subset\mathcal{A}_n\rtimes p^t\mathbb{Z}$. There is a subnormal $p$-chain $\mathcal{A}_n\rtimes p^j\mathbb{Z}$, $0\leq j\leq t$ between $\mathcal{A}_n\rtimes p^t\mathbb{Z}$ and $\mathcal{L}_n$, thus it is left to construct the subnormal $p$-chain between $V$ and $\mathcal{A}_n\rtimes p^t\mathbb{Z}$. Identifying $\mathcal{A}_n\rtimes p^t\mathbb{Z}$ with $\mathcal{L}_{np^t}$ we reduce to the case $t=0$. Now, since $V$ is weakly maximal, $\mathcal{A}_n/V_0$ is isomorphic to $R$ as $R$ modules, and denote the corresponding map $\mathcal{A}_n\rightarrow R$ by $\psi$. Then $V_i=\psi^{-1}{((1+x)^iR)}$ form a $p$-chain between $V_0$ and $\mathcal{A}_n$, and hence groups with triples $(1,V_i,v)$ form subnormal $p$-chain between $V$ and $\mathcal{L}_n$, where $v$ is a vector from a triple of $V$. Hence there is a $p$-chain that intersects to $V$ and thus $V$ is closed in pro-$p$ topology by the Proposition~\ref{pro2}. 
\end{proof}

We recall the following definition.
\begin{definition}(see \cite{Grig})
The group is called hereditary free from finite normal subgroups if any subgroups of it of finite index have only infinite nontrivial normal subgroups.
\end{definition}
\begin{corollary}\label{hered}
Any nontrivial normal subgroup of $\mathcal{L}_n$ is infinite. Therefore, $\mathcal{L}_n$ is hereditary free from finite normal subgroups.  
\end{corollary}
Therefore we can derive from \cite{Grig} the following result:
\begin{corollary}\label{topfree}
Any faithful action of  $\mathcal{L}_n$  on the rooted tree induces topologically free action on the
boundary of the tree.
\end{corollary}

\subsection{Growth of normal subgroups of $\mathcal{L}_{n}$}
Recall that the zeta function for normal subgroups of $\mathcal{L}_n$ is the following:
\begin{equation}
\zeta_n(z)=\sum_{N\lhd\mathcal{L}_n:[\mathcal{L}_n:N]<\infty}[\mathcal{L}_n:N]^{-z}
\end{equation}
We have the following formula:
\begin{theorem}\label{norm1}
\begin{equation*}
\begin{aligned}
& \zeta_n(z)=\sum_{s\geq 1}s^{-z}\prod_{r\in R, prime, r|1-x^s, r\neq 1-x}\left(\sum_{M: M\in{Mod}_r, (1-x^s)R^n\subset M}[R^n:M]^{-z}\right)\times \\&\left(\sum_{M:M\in{Mod}_{1-x}, (1-x^s)R^n\subset M}\tau(M)[R^n:M]^{-z}\right),
\end{aligned}
\end{equation*}
where ${Mod}_r$ is the set of all submodules $M$ in $R^n$ such that $r^kR^n\subset M$ for some $k$, and, for $M\in{Mod}_{1-x}$,  $\tau(M)=|\ker \varphi|$ for $\varphi:R^n/M\rightarrow R^n/M$, $\varphi(r)=(1-x)r$.
\end{theorem}
\begin{proof}
Recall that if $M$ is a submodule of $R^n$ of finite index, then $R^n/M=\oplus_{r\in R, r\text{ prime }} N_r$ is the primary decomposition of $R^n/M$ (Theorem 6.7 in \cite{H}), with only finite number of $N_r$ not zero. Let $\pi:R^n\rightarrow R^n/M$ and let $M_r=\pi^{-1}(N_r)$. Then $M_r\in{Mod}_r$, $M=\cap_{r}M_r$, and $[R^n:M]=\prod_{r}[R^n:M_r]$. We obtain
\begin{equation*}
\begin{aligned}
&\zeta_n(z)=\sum_{s\geq 1, M\subset R^n v\in R^n/M: (1-x^s)R^n\subset M, (1-x)v=0}(s[R^n:M])^{-z}=\\&\sum_{s\geq 1}s^{-z}\prod_{r\in R, prime, r|1+x^s, r\neq 1-x}\left(\sum_{M: M\in{Mod}_p, (1-x^s)R^n\subset M}[R^n:M]^{-z}\right)\times \\&\left(\sum_{M:M\in{Mod}_{1-x}, (1-x^s)R^n\subset M}\tau(M)[R^n:M]^{-z}\right)
\end{aligned}
\end{equation*}
 \end{proof}
\begin{corollary}\label{bounddd}
\begin{equation*}
\begin{aligned}
&\zeta_n(z)\preceq p^n\zeta(z)\xi_n(z)
\end{aligned}
\end{equation*}
where 
\[
\zeta(z)=\sum_{s\geq 1}s^{-z}
\]
 is the Riemann zeta function,
\[
\xi_n(z)=\sum_{M\subset_f R^n}[R^n:M]^{-z}
\]
where the summation is over submodules of finite index in $R^n$, is the zeta function of the module $R^n$, and 
\[
\sum_{m\geq 1}b_m m^{-z}\preceq\sum_{m\geq 1}c_m m^{-z} 
\]
if and only if $0\leq b_m\leq c_m$ for all $m\geq 1$.
\end{corollary}
\begin{proof}
Note that for  $M\in {Mod}_{1-x}$, we have that $R^n/M=\oplus_1^n R/(1-x)^{k_i}R$,  therefore $\dim_{\mathbb{F}_p} \ker \phi\leq n$ and thus $\tau(M)\leq p^n$. Note also that 
\[
\xi_n(z)=\sum_{M\subset_f R^n}[R^n:M]^{-z}=\prod_{r\in R, r\text{ prime }}\left(\sum_{M\in{Mod}_r}[R^n:M]^{-z}\right),
\]
and thus 
\[
\sum_{M: M\in{Mod}_r, (1-x^s)R^n\subset M}[R^n:M]^{-z}\preceq\sum_{M\in{Mod}_r}[R^n:M]^{-z}, 
\]
and we obtain the first formula. 

%The proof of the second formula follows line by line from the second proof of (15.10) in \cite[subgroupgrowth], p. 308.
\end{proof}

\begin{corollary}\label{norm2}
Let $s^{\triangleleft}_m(\mathcal{L}_n)$ be the number of normal subgroups of $\mathcal{L}_n$ of index smaller or equal to $m$. Then 
\[
\limsup_{m}\frac{\log s^{\triangleleft}_m(\mathcal{L}_n)}{\log m}\leq n
\]
 \end{corollary}
\begin{proof}
Notice that by the lemma \ref{submodulegrowth} 
\[
\xi_n(z)=1+\sum_{t\geq 1}\frac{p^{tn}-p^{(t-1)n}}{p^{tz}}=\frac{p^z-1}{p^z-p^n},
\]
so $\xi_n(z)$ is analytic for $\Re z>n$, where $\Re z$ is the real part of $z$, and thus by  the corollary \ref{bounddd} 
$\zeta_n(z)$ is analytic for $\Re z>n$, since Riemann zeta function is analytic for $\Re z>1$. Now observe that 
\[
\zeta_n(z)=\sum_{m=1}^{\infty}\frac{a^{\triangleleft}_m(\mathcal{L}_n)}{m^z},
\]
where $a^{\triangleleft}_m(\mathcal{L}_n)$ is the number of normal subgroups of $\mathcal{L}_n$ of index $m$, and note that
$s^{\triangleleft}_m(\mathcal{L}_n)=\sum_{j=1}^{m}a^{\triangleleft}_j(\mathcal{L}_n)$, from which we obtain the statement of the corollary
 (see the beginning of chapter 15 of \cite{LS}).
\end{proof}

\begin{remark}
Note that to derive the estimate Corollary~\ref{norm2} we only used the property that if $V$ is a normal subgroup of $\mathcal{L}_n$ then $V\cap\mathcal{A}$ is invariant under the shift, no other properties of normal subgroups listed in the Lemma~\ref{normal_group} were used. 
\end{remark}

\subsection{Subgroups, normal subgroups and factors of $\mathcal{L}_{1,p}$.}
When $n=1$ we give one more description of normal subgroups.

\begin{lemma}\label{one_more}
Let $e\in\mathcal{A}_1$ be the element with $1$ on the position $0$ and zeros everywhere else. Then to each nontrivial normal subgroup $V$ of $\mathcal{L}_1$ uniquely corresponds triple $(s,f,h)$ where $s\in\mathbb{N}$ is defined as before, and $f,h\in\mathbb{F}_p[x]$ are defined in the following way: $f$ is the polynomial of minimal degree such that $f(x)e\in{}V$ and $f(0)=1$, and $h$ is the polynomial such that $\deg h<\deg f$ and $h(x)e=v\mod V_0$. Moreover, $f(0)=1$. If $s\geq 1$ then $f$ divides $1-x^s$, and $h$ is either $0$ or $h=f(1-x)^{-1}$ (the latter is possible only if $f$ is divisible by $1-x$).  
\end{lemma}

\begin{proof}
Since $V$ is normal and not trivial then $V_0$ is also not trivial. Indeed if $(v,s)\in V$ then the commutator $[(w,0),(v,s)]=(w-x^s w,0)\in V_0$. Note that if we identify $\mathcal{A}_1$ and $R$, with $e$ identified to $1\in R$, then each nontrivial subgroup $V_0$ of $\mathcal{A}_1$ such that $x{}V_0=V_0$ is in fact a nontrivial ideal of $R$, which must be of the form $fR$ for some polynomial $f\in\mathbb{F}_p[x]$ such that $f(0)=1$. Under this identification $\mathcal{A}_1/V_0$ is identified with $R/fR=\mathbb{F}_p[x]/f\mathbb{F}_p[x]$, and since $v$ from the triple is defined up to addition of element of $V_0$, we choose the unique polynomial of the minimal degree $h$ which is the representative of the corresponding coset of $\mathbb{F}_p[x]/f\mathbb{F}_p[x]$.

Now let $V$ be the normal subgroup of $\mathcal{L}_1$, $(s,V_0,v)$, $(s,f,h)$ be the triples associated to it as defined in the statement of the lemma. We already proved that $(1-x^s)\mathcal{A}_1\subset{}V_0$, which implies that $(1-x^s)e\in{}V_0$ and thus by the definition of $f$ that $1-x^s$ is divisible by $f$. Now, suppose $h\neq{0}$. On the one hand the degree of $h$ is less then the degree of $f$, since otherwise we reduce it preserving the requirements on $h$. On the other hand it follows from $(1-x)v\in{}V_0$ and $v=h(x)e\mod V_0$ that $(1-x)h(x)e\in{}V_0$, and thus that $f$ divides $(1-x)h$. Therefore $f=(1-x)h$.
\end{proof}

To summarize, we have the following three cases for the triple $(s,f,h)$ associated to the normal subgroup $V$ of $\mathcal{L}$: either $s=0$ and then the triple is $(0,f,0)$, with no restrictions on $f$ except $f(0)=1$; or $s\geq{}1$, then necessarily $f\neq{}0$, and either $h=0$ and the triple is $(s,f,0)$ with $f|(1-x^s)$ or $h\neq{}0$ and the triple is $(s,(1-x)h,h)$ with $h$ dividing $(1-x^s)(1-x)^{-1}=1+x+\dots+x^{s-1}=\varphi_s(x)$. 

\begin{definition}\label{def_norm}
For $f\in{}\mathbb{F}_p[x]$ and $s\geq{1}$, such that $f(0)=1$ and $f|1-x^s$ denote by $B_{f,s}$ the corresponding subgroup with projection onto $\mathbb{Z}$ equal to $s\mathbb{Z}$, intersection with base equal to $fR$ and the corresponding vector equal to $0$. If moreover $1-x|f$, that is if $f(1)=0$, denote the group with triple $(s,f,f(1-x)^{-1})$ by $C_{f,s}$.
\end{definition}

\begin{theorem}\label{prof}
Any nontrivial normal subgroup of $\mathcal{L}_1$ of finite index is either $B_{f,s}$ or $C_{f,s}$ for some polynomial $f$ and $s\geq 1$. We have the following description of the lattice of normal subgroups of finite index in $\mathcal{L}_1$.
\begin{itemize}
\item $B_{f,s}\subset{}B_{f',s'}$ if and only if $s'|s$, $f'|f$. 
\item $C_{f,s}\subset{}C_{f',s'}$ if and only if $s'|s$, $f'|f$ and \\ 
$f=\varphi_{s/s'}(x^{s'})f'\mod{}(1-x)f'$. 
\item $C_{f,s}\subset{}B_{f',s'}$ if and only if $s'|s$ and $f'|f(1-x)^{-1}$.
\item $B_{f,s}\subset{}C_{f',s'}$ if and only if $s'|s$, $f'|f$ and $s/s'$ is even.
\end{itemize}
In particular, the set of groups $B_{1-x^s,s}$, $s\geq 1$, is a basis for profinite topology.
\end{theorem} 
\begin{proof}
By the Lemma~\ref{one_more} to any nontrivial normal subgroup $V$ corresponds a triple $(s,f,h)$. If $s=0$ then $V\subset\mathcal{A}_1$ and therefore $V$ has infinite index. If $s\geq 1$ then the index of $V$ is $sp^{\deg f}$ and by Definition~\ref{def_norm} $V=B_{f,s}$ if $h=0$ and $V=C_{f,s}$ if $h\neq 0$.
Inclusions follow from the Lemma~\ref{group}. To prove the last statement note that $B_{f,s}\supset{}B_{1-x^s,s}$, since $f|1-x^s$ by the definition of $B_{f,s}$. Also, $C_{f,s}\supset{}B_{1-x^{ps},ps}$, by the last item and the fact that $f|(1-x^s)^p=1-x^{ps}$. Thus any normal subgroup of finite index contains some $B_{1-x^s,s}$, therefore it is a basis for profinite topology.
\end{proof}
Now we describe the finite factors of $\mathcal{L}_1$.
\begin{theorem}
Any extension $H$ of a finite cyclic group $G$ by a finite elementary p-group $E$, such that the induced action of $G$ on $E$ has a cyclic vector, is a factor of $\mathcal{L}_1$. In particular, if G has order $s$ and $m$ is a cyclic vector of $E$ with minimal polynomial $f$ (we normalize $f$ so that $f(0)=1$) then any kernel of a surjective homomorphism $\mathcal{L}_1\rightarrow{}H$ is either $B_{f,s}$ or $C_{f,s}$. Moreover, if we fix action of $G$ on $E$ then there are at most two non-equivalent extensions. If $f(1)=1$ or if $f(1)=0$ but $f$ does not divide $\varphi_s(x)$ then there is only split extension, in all other cases there are exactly two extensions.
\end{theorem}
\begin{proof}
 $\mathcal{L}_1$ can be given by $\{a,b|b^p=1,[b^{a^n},b^{a^m}]=1\text{ for all }n,m\}$. Therefore if $H$ is an extension of $G$ by $E$, $G$ is generated by $x$ and $m\in{}E$ is a cyclic vector, then the map $b\mapsto{}m$, $a\mapsto{}x$ can be lifted to a surjective map $\mathcal{L}\rightarrow{}H$. 

The number of equivalent classes of extensions of $G$ by $E$ is equal to the order of $H^2(G,E)$, where $H^2(G,E)$ is the second cohomology group of $G$ with coefficients in $G$-module $E$. By \cite{brown} $H^2(G,E)=E^G/NE$, where $E^G$ is the set of fixed points in $E$ under the action of $G$ and $N=1+x+\dots+x^{s-1}$. If $E$ has a cyclic vector $m$ than it is isomorphic to $\mathbb{F}_p[x]/f\mathbb{F}_p[x]$ where $f$ is the minimal polynomial for $m$. Thus if $g+(f)$ is a fixed point it follows that $xg=g\mod{}f$, or equivalently $f|(1-x)g$. Since $\deg{}g<\deg{}f$ this can only happen if $f=(1-x)g$, i.e. if $f(1)=0$ and $g=f(1-x)^{-1}$.  Thus $\dim E^G=1$ if $f(1)=0$ and $0$ otherwise. Also, $NE=\{1+\dots+x^{s-1}+(f)\}$, and thus if $f(1)=0$, $NE=E^G$ if and only if $f|(1+\dots+x^{s-1})$.
\end{proof}

\subsection{Profinite and pro-$p$ completion of $\mathcal{L}_{n,p}$.}
Perhaps the following lemma is well-known.
\begin{lemma}\label{prof1}
The profinite completion of $\mathcal{L}_n $ is $\mathbb{F}_p[[\hat{\mathbb{Z}}]]^n\rtimes\hat{\mathbb{Z}}$, where $\hat{\mathbb{Z}}$ is the profinite completion of $Z$ and $\mathbb{F}_p[[\hat{\mathbb{Z}}]]$ is the profinite completion of the ring $R=\mathbb{F}_p[\mathbb{Z}]$, which is also the projective limit $\varprojlim{}\mathbb{F}_p[\mathbb{Z}/n\bz]$.
\end{lemma}
\begin{proof}
First recall that any ideal of $R$ is of the form $fR$ where $f\in\mathbb{F}_p[x]$ and $f(0)=1$. Now, any such irreducible $f$ divides $1-x^n$ for some $n$. Since $f$ is the product of irreducibles it follows that for any $f$ there is $n$ such that $f|1-x^n$. Therefore the profinite completion of the ring $R=\mathbb{F}_p[x]$ is $\varprojlim{}R/(1-x^n)R=\varprojlim{}\mathbb{F}_p[\mathbb{Z}/n\mathbb{Z}]$.

Consider now surjective maps  $\psi_m:\mathbb{F}_p[[\hat{\mathbb{Z}}]]^n\rtimes\hat{\mathbb{Z}}\rightarrow{}\mathbb{F}_p[\mathbb{Z}/m\mathbb{Z}]^n\rtimes\mathbb{Z}/m\mathbb{Z}=\mathcal{L}_n/N_m$, where the normal subgroup $N_m$ has triple $(m,(x-1)^m\mathcal{A}_n,0)$. The maps $\psi_m$ are compatible with the projective system $\mathcal{L}_n/N_m$ and therefore induce a map $\psi:\mathbb{F}_p[[\hat{\mathbb{Z}}]]^n\rtimes\hat{\mathbb{Z}}\rightarrow{}\varprojlim\mathcal{L}_n/N_m$. Then $\psi$ is surjective by the Corollary~$1.1.6$ in \cite{Zal}, since $\mathbb{F}_p[[\hat{\mathbb{Z}}]]^n\rtimes\hat{\mathbb{Z}}$ is compact. On the other hand $\ker\psi_m=\ker\pi_m\times{}m\hat{\mathbb{Z}}$, where $\pi_m:\mathbb{F}_p[[\hat{\mathbb{Z}}]]^n\rightarrow{}\mathbb{F}_p[\mathbb{Z}/m\mathbb{Z}]^n$ is the natural projection. Since the intersection of $\ker\pi_m$ is trivial, the intersection of $\ker\psi_m$ is also trivial which implies that $\psi$ is an isomorphism. 

It is left to show that normal subgroups $N_m$ form a basis of identity for the group $\mathcal{L}_n$, but it follows from Lemmas~\ref{3_1} and \ref{normal_group} that if $V$ is normal subgroup of finite index in $\mathcal{L}_n$ and $s>0$ the generator of its projection on $\bz$ then the normal subgroup with a triple $(sp,(x^{sp}-1)\mathcal{A}_n,0)$ lies in $V$. 
\end{proof}

%\begin{theorem}

%\end{theorem}

\begin{theorem}\label{closedd}
Let $V$ be the subgroup of $\mathcal{L}_n$ with a triple $(s,V_0,v)$. Then it is profinitely closed in $\mathcal{L}_n$ if and only if either $s>0$ or $V=V_0$ is closed as a subset of $\mathbb{F}_p[[\hat{\mathbb{Z}}]]^n$ (that is, $V$ is equal to the intersection of its closure with $\mathcal{A}_n$).
\end{theorem}
\begin{proof}
If $s>0$ then $V$ is finitely generated by Theorem~\ref{3_3} and hence closed by Corollary~\ref{3_4}. If $s=0$ and $V=V_0\subset\mathcal{A}_n$ which we identify with $R^n$, let $\bar{V}$ denote the closure of $V$ in  $\mathbb{F}_p[[\hat{\mathbb{Z}}]]^n$. Then since $\mathbb{F}_p[[\hat{\mathbb{Z}}]]^n$ is compact $\bar{V}$ is also compact and hence is still closed in $\mathbb{F}_p[[\hat{\mathbb{Z}}]]^n\rtimes\hat{\mathbb{Z}}$. Since $\mathbb{F}_p[[\hat{\mathbb{Z}}]]^n\cap\mathcal{L}_n=\mathbb{F}_p[[\hat{\mathbb{Z}}]]^n\cap R^n$, and hence $\bar{V}\cap\mathcal{L}_n=\bar{V}\cap R^n=V$.
\end{proof}

\begin{corollary}
There are subgroups of $\mathcal{L}_1$ which are not closed in the profinite topology.
\end{corollary}
\begin{proof}
 Here is an example due to D.~Segal. Consider in $\mathcal{A}_1=\sum_\mathbb{Z}(\mathbb{Z}/p\mathbb{Z})$ the subgroup $H$ such that $v\in H$ if and only if $v_i=0$ for $i$ not a prime number. For any $n$ the map $\mathbb{F}_p[\mathbb{Z}]\rightarrow\mathbb{F}_p[\mathbb{Z}/n\mathbb{Z}]$ maps $H$ onto $\mathbb{F}_p[\mathbb{Z}/n\mathbb{Z}]$. Thus the profinite closure of $H$ is $\mathbb{F}_p[[\hat{\mathbb{Z}}]]$, and so $H$ is not closed (that is, the intersection of the closure of $H$ with $\mathcal{L}_1$ is equal to $\mathcal{A}_1\neq H$).    
\end{proof}

We will now study the pro-$p$ completion of $\mathcal{L}_1$. The index of $B_{f,s}$ or $C_{f,s}$ is equal to $p^{\deg{f}}s$, therefore we need to consider the cases when $s$ is a power of $p$. In this case $1-x^{p^n}=(1-x)^{p^n}$, therefore $f=(1-x)^i$ for $0\leq{}i\leq{}p^n$. Thus we obtain a description of normal subgroups with index a power of $p$:

\begin{definition}
Define $B'_{i,n}=B_{(1-x)^i,p^n}$, $0\leq{}i\leq{}p^n$ and $C'_{i,n}=C_{(1-x)^i,p^n}$, $1\leq{}i\leq{}p^n$.  We have that $B'_{i,n}$ and $C'_{i,n}$ is the list of all normal subgroups of $\mathcal{L}_1$ with index a power of $p$.
\end{definition}

The following Theorem describes the lattice of normal subgroups of $\mathcal{L}_1$ of index a power of $p$.

\begin{theorem}
\begin{itemize}
\item $B'_{i,n}\subset{}B'_{j,m}$ if and only if $m\leq{}n$ and $j\leq{}i$. 
\item $C'_{i,n}\subset{}C'_{j,m}$ if and only if either $m<n$, $j<i$ or $m=n$ and $j=i$. 
\item $C'_{i,n}\subset{}B'_{j,m}$ if and only if $m\leq{}n$ and $j<i$.
\item $B'_{i,n}\subset{}C'_{j,m}$ if and only if $m<n$ and $j\leq{}i$.
\end{itemize}
In particular, the set of groups $B'_{p^n,n}=B_{1-x^{p^n},p^n}$ is a basis for the pro-$p$ topology.
\end{theorem} 
\begin{proof}
Follows from Theorem~\ref{prof}
\end{proof}
%Observe that a subgroup of finite index is open and closed in t
\begin{corollary}\label{pro-p closed}
The subgroup $V$ of finite index in $\mathcal{L}_1$ is pro-$p$ closed \Iff its triple $(s,V_0,v)$ satisfies that $s=p^t\bz$ for some $t$ and $V_0\supset (1-x^{p^n})\mathcal{A}_1$ for some $n\geq t$. 
\end{corollary}

\begin{lemma}\label{3_19}
The pro-$p$ completion of $\mathcal{L}_n$ is $\mathbb{F}_p[[{\mathbb{Z}_p}]]^n\rtimes{\mathbb{Z}_p}$, where ${\mathbb{Z}_p}$ is the pro-$p$ completion of $\mathbb{Z}$ and $\mathbb{F}_p[[{\mathbb{Z}_p}]]$ is the pro-$p$ completion of the ring $R=\mathbb{F}_p[\mathbb{Z}]$, which is also the projective limit $\varprojlim{}\mathbb{F}_p[\mathbb{Z}/p^m\mathbb{Z}]$.
\end{lemma}
\begin{proof}
The proof repeats almost verbatim the proof of Lemma~\ref{prof1}, with $\hat{\bz}$ replaced by $\bz_p$.
\end{proof}

\begin{lemma}\label{need0}
Let $V$ be a subgroup of $\mathcal{L}_n$ with a triple $(s,V_0,v)$ and $s>0$. Then $V$ is closed in the pro-$p$ topology if and only if $V_0$ is equal to the intersection of its closure in $\mathbb{F}_p[[\mathbb{Z}_p]]^n$ with $R^n$, and for any $0<q<s$ such that $|q|_p\geq |s|_p$ 
\[
\frac{x^q-1}{x^s-1}v\not\in \mathcal{A}_n+\bar{V_0},
\]
where $\bar{V_0}$ is the closure of $V_0$ in $\mathbb{F}_p[[\mathbb{Z}_p]]^n$ and $|q|_p$ is the highest power of $p$ that divides $q$.
\end{lemma}
\begin{proof}
%Note that the condition $(x^q-1)v\in (x^s-1)\mathcal{A}_n+V_0$ is equivalent to 
%\[
%\frac{x^q-1}{x^s-1}v\in \mathcal{A}_n+\bar{V_0},
%\]

Let $\bar{V}$ be the closure of $V$ in $\mathbb{F}_p[[{\mathbb{Z}_p}]]^n\rtimes{\mathbb{Z}_p}$, and $\bar{V_0}$ the closure of $V_0$ in $\mathbb{F}_p[[{\mathbb{Z}_p}]]^n$. Then $\bar{V_0}=\bar{V}\cap\mathbb{F}_p[[{\mathbb{Z}_p}]]^n$, hence if $V=\bar{V}\cap\mathcal{L}_n$ then \[V_0\subset\bar{V_0}\cap\mathcal{A}_n\subset\bar{V}\cap\mathcal{A}_n\subset V\cap\mathcal{A}_n=V_0.\]

Suppose that $V_0$ is closed in $\mathbb{F}_p[[{\mathbb{Z}_p}]]^n$. 
Suppose that for any $0<q<s$ such that $|q|_p\geq |s|_p$ 
\[
\frac{x^q-1}{x^s-1}v\not\in \mathcal{A}_n+\bar{V_0}.
\]
Let $(w,c)\in\mathcal{L}_n$ be some element in  $\bar{V}$. It means that there is a sequence $u_i\in V_0$ and $k_i\in\bz$ such that $(u_i+\varphi_{k_i}(x^s)v,k_is)$ tends to $(w,c)$, or $(u_i+\varphi_{k_i}(x^s)v-x^{k_is-c}w,k_is-c)$ tends to identity. This happens if and only if 
\[
u_i+\varphi_{k_i}(x^s)v-x^{k_is-c}w\to 0\text{ in }\mathbb{F}_p[[{\mathbb{Z}_p}]]^n,
\]  
and $k_is-c\to 0$ in $\bz_p$, which means that one can write $k_is=p^{m_i}r_i+c$ and $m_i\to\infty$ as $i\to\infty$. In follows in particular that $|c|_p\geq |s|_p$. Note that $x^{k_is-c}=x^{p^{m_i}r_i}\to 1$ in $\mathbb{F}_p[[{\mathbb{Z}_p}]]=\varprojlim{}\mathbb{F}_p[\mathbb{Z}/p^m\mathbb{Z}]$ since $x^{p^{m_i}r_i}=1$ in $\mathbb{F}_p[\mathbb{Z}/p^m\mathbb{Z}]$ for $m\leq m_i$. For the same reason, 
\[
\varphi_{k_i}(x^s)=\frac{x^{k_is}-1}{x^s-1}=\frac{x^{p^{m_i}r_i}x^{c}-1}{x^{s}-1}\to\frac{x^c-1}{x^s-1},
\]
considered as a sequence in the field of fractions of $\mathbb{F}_p[[{\mathbb{Z}_p}]]$. It follows that $u_i$ also has some limit $u\in\bar{V_0}$ and we have the equality
\[
u+\frac{x^c-1}{x^s-1}v-w=0.
\]
Let $c=ks+q$ for $0\leq q<s$. Since $|c|_p\geq |s|_p$ we have that $|q|_p\geq |s|_p$. We also have that $x^c-1=(x^s-1+1)^kx^q-1$, and therefore $x^c-1=g(x)(x^s-1)+x^q-1$ for some polynomial $g$. Thus 
\[
\frac{x^q-1}{x^s-1}v=\frac{x^c-1}{x^s-1}v-g(x)v=w-g(x)v-u\in R^n+\bar{V_0},
\]
which is by assumption impossible unless $q=0$. Thus $c=ks$ and $w=u+\frac{x^{ks}-1}{x^s-1}v$, from which it follows that $u\in\bar{V_0}\cap\mathcal{A}_n=V_0$ and therefore 
$(w,c)=(u,0)(v,s)^k\in V$.

Conversely, suppose that there is $0<q<s$ such that $|q|_p\geq |s|_p$ and
\[
\frac{x^q-1}{x^s-1}v=w+u,
\]
for $w\in \mathcal{A}_n$ and $u\in\bar{V_0}$. Choose a sequence $u_i\in V_0$ that tends to $u$, and a sequence $k_i\in\bz$ such that $k_is\to q$ in $\bz_p$ (it is possible to find since $|q|_p\geq |s|_p$). Then $(u_i,0)(v,s)^{k_i}\in V$ but tend to $(w,q)$ by reasoning similar to previous considerations. Hence $V$ is not closed.
\end{proof}

\begin{corollary}\label{need}
$\mathbb{F}_p[[\mathbb{Z}_p]]$ is isomorphic to $\mathbb{F}_p[[t]]$, the ring of formal power series in $t$ over $\mathbb{F}_p$. The map 
\[
a=\sum_{i\geq{}0}a_ip^i\mapsto(1+t)^a=\prod_{i\geq{}0}(1+t^{p^i})^{a_i},
\]
where $a_i\in\{0,1,\dots,p-1\}$, is a well-defined isomorphism of $\mathbb{Z}_p$ into the multiplicative group of $\mathbb{F}_p[[t]]$ and the action of $\mathbb{Z}_p$ on $\mathbb{F}_p[[t]]$ in $\mathbb{F}_p[[{\mathbb{Z}_p}]]^n\rtimes{\mathbb{Z}_p}$ ($\mathbb{F}_p[[{\mathbb{Z}_p}]]$ identified with $\mathbb{F}_p[[t]]$) in Lemma~\ref{3_19} is given by the formula $a*f(t)=(1+t)^af(t)$.
\end{corollary}
\begin{proof}
Note that any element of $\mathbb{F}_p[[\mathbb{Z}_p]]=\varprojlim{}R/(1-x^{p^n})R=\varprojlim R/(1-x)^{p^n}R=\varprojlim_j R/(1-x)^jR$ can be written as $\sum_{i\geq{}0}\varepsilon_i(x-1)^i$, for $\varepsilon_i\in\mathbb{F}_p$. Then the map $\sum_{i\geq{}0}\varepsilon_i(x-1)^i\mapsto\sum_{i\geq{}0}\varepsilon_it^i$ gives an isomorphism of rings  $\varprojlim{}R/(1-x)^i R$ and $\mathbb{F}_p[[t]]$. The other claims are obvious.
\end{proof}

\begin{lemma}\label{need1}
Let $g=(x-1)^ch$ for some $g,h\in R$, where $x-1$ does not divide $h$. Then $g\mathbb{F}_p[[\mathbb{Z}_p]]\cap R=(x-1)^cR$.
\end{lemma}
\begin{proof}
Note that multiplying by some positive power of $x$ we may assume that $g,h\in\mathbb{F}_p[x]$. Hence it suffices to prove that if $g\in \mathbb{F}_p[x]$, we have that $g\mathbb{F}_p[[\mathbb{Z}_p]]\cap \mathbb{F}_p[x]=(x-1)^c\mathbb{F}_p[x]$. Change variable $t=x-1$ and note that $\mathbb{F}_p[x]=\mathbb{F}_p[t]$.  Now, using first part of Corollary~\ref{need} we have to prove that if $g\in\mathbb{F}_p[t]$, $g=t^ch$ and $h(0)\neq 0$ then $g\mathbb{F}_p[[t]]\cap t^c\mathbb{F}_p[t]$. Notice that since $h(0)\neq 0$, $h$ is invertible in $\mathbb{F}_p[[t]]$ and therefore $g\mathbb{F}_p[[t]]=t^c\mathbb{F}_p[[t]]$. It is left to recall that the element  $\sum_{j\geq 0} \ep_j t^j\in\mathbb{F}_p[[t]]$ belongs to $\mathbb{F}_p[t]$ if nad only if the number of nonzero $\ep_j$ is finite.
\end{proof}

\begin{lemma}\label{need2}
Let $U\subset R^n$ be a subgroup such that $e(U)=p^e$ for some $e\geq 0$. Let $k=np^e$. Then there are elements $h_1,\dots, h_k\in R^n$ and $d_1,\dots,d_k\in \mathbb{F}_p[x]$ such that $d_j=0$ or $d_j(0)=1$,  $R^n=\oplus_{j=1}^{k}R^{p^e}h_j$ and $U=\oplus_{j=1}^{k}(Rd_j)^{p^e}h_j$. Moreover, $U$ is equal to the intersection of its closure in $\mathbb{F}_p[[\mathbb{Z}_p]]^n$ with $R^n$ if and only if ${\det}^*U=\prod_{j:d_j\neq 0} d_j$ is a power of $1-x$ (recall the definitions \ref{ddet} and \ref{expp}).
\end{lemma}
\begin{proof}
The existence of such $h_j$ and $d_j$ follows using the fact that $R^n$ seen as $R$ module with the action $x*v=x^{p^e}v$ is isomorphic to $R^{k}$ (to see this, consider a formula similar to \eqref{isomorphism}) and then applying Lemma~\ref{decomp}.

To prove the second part note first that $\mathbb{F}_p[[\mathbb{Z}_p]]^n=\oplus_{j=1}^{k}(\mathbb{F}_p[[\mathbb{Z}_p]])^{p^e}h_j$. Indeed, suppose that $u_j\in \mathbb{F}_p[[\mathbb{Z}_p]]$ and $\sum_j u_j^{p^e}h_j=0$. Let $u_j=\lim_q u_{j,q}$ so that $u_{j,q}\in R$ and $u_j-u_{j,q}\in (x-1)^q \mathbb{F}_p[[\mathbb{Z}_p]]$. It follows that $\sum_j u_{j,q}^{p^e}h_j\in (x-1)^q \mathbb{F}_p[[\mathbb{Z}_p]]\cap R^n=(x-1)^q R^n$ by Lemma~\ref{need1}. Hence $\sum_j u_{j,q}^{p^e}h_j=(x-1)^q\sum_j v_{j,q}^{p^e}h_j$ for some $v_{j,q}\in R$. Thus for $q>p^e$ we obtain $\sum_j (u_{j,q}-(x-1)^{q-p^e}v_{j,q})^{p^e}h_j=0$. It follows that $u_{j,q}=(x-1)^{q-p^e}v_{j,q}$ and therefore $u_j=\lim_q u_{j,q}=0$.

It follows that if $\bar{U}$ is the closure of $U$ in $\mathbb{F}_p[[\mathbb{Z}_p]]^n$, then \[\bar{U}=\oplus_{j=1}^{k}(\mathbb{F}_p[[\mathbb{Z}_p]]d_j)^{p^e}h_j.\]
Let $d_j=(x-1)^{r_j}h_j$ for some $r_j\geq 0$ and $h_j$ such that $x-1$ does not divide $h_j$. Then by Lemma~\ref{need1} we obtain that 
\begin{equation*}
\begin{aligned}
&\bar{U}\cap R^n=\oplus_{j=1}^{k}(\mathbb{F}_p[[\mathbb{Z}_p]]d_j)^{p^e}h_j\cap R^n=\\&\oplus_{j=1}^{k}(\mathbb{F}_p[[\mathbb{Z}_p]](x-1)^{r_j})^{p^e}h_j\cap R^n=\oplus_{j=1}^{k}(R(x-1)^{r_j})^{p^e}h_j,
\end{aligned}
\end{equation*}
hence $U=\bar{U}\cap R^n$ if and only if $h_j=1$.
\end{proof}

\begin{theorem}\label{closeddp}
Let $V$ be a subgroup of $\mathcal{L}_n$ with a triple $(s,V_0,v)$. Then $V$ is closed in the pro-$p$ topology if and only if either $s=0$ and $V_0$ is closed as a subset of $\mathbb{F}_p[[\mathbb{Z}_p]]^n$, or  $s>0$ and for $V_0$ we have that $e(V_0)$ divides $p^{|s|_p}$, ${\det}^*V_0$ is a power of $1-x$, and for any $0<q<s$ such that $|q|_p\geq |s|_p$ 
\[
(x^q-1)v\not\in (x^s-1)\mathcal{A}_n+(x-1)^{p^{|s|_p}}V_0.
\]
%where $|q|_p$ is the highest power of $p$ that divides $q$.
\end{theorem}
\begin{proof}
The case $s=0$ is obvious, so assume $s>0$.
If $V$ is pro-$p$ closed then it follows that $V$ is the intersection of the $p$-chain $V_j$ with triples $(s_j,V_j,v_j)$. Since it is $p$-chain, each $s_j$ is a power of $p$, and since $s_j|s$, the sequence $s_j$ stabilize. Thus eventually $s_j=p^{e'}$ for some $e'\leq e$. Then eventually $x^{p^{e}}V_j=V_j$, and since $V_0=\cap V_j$, we obtain that also $x^{p^e}V_0=V_0$. Then since $V$ is pro-$p$ closed, $V_0$ is closed in $\mathbb{F}_p[[\mathbb{Z}_p]]^n$ and therefore by Lemma~\ref{need2} ${\det}^*V_0$ is a power of $1-x$.

Conversely, suppose that $x^{p^e}V_0=V_0$  and ${\det}^*V_0$ is a power of $1-x$. By Lemma~\ref{need2} $V_0$ is closed in $\mathbb{F}_p[[\mathbb{Z}_p]]^n$. 

Notice also that 
\[
\frac{x^q-1}{x^s-1}v\not\in R^n+\bar{V_0}
\]
if and only if $(x^q-1)v\not\in (x^s-1)R^n+((x^s-1)\bar{V_0})\cap R^n$. Note that we can write $x^s-1=((x-1)h')^{p^e}$ where $h'$ is not divisible by $x-1$, and therefore invertible in $\mathbb{F}_p[[\mathbb{Z}_p]]$. Using the decomposition from the Lemma~\ref{need2} we have that $\bar{V_0}=\oplus_{j=1}^{k}(\mathbb{F}_p[[\mathbb{Z}_p]]d_j)^{p^e}h_j$, hence 
\[(x^s-1)\bar{V_0}\cap R^n=\oplus_{j=1}^{k}(\mathbb{F}_p[[\mathbb{Z}_p]](x-1)d_j)^{p^e}h_j\cap R^n=\oplus_{j=1}^{k}(R(x-1)d_j)^{p^e}h_j.\] 
It follows that 
\begin{equation*}
\begin{aligned}
& ((x^s-1)\bar{V_0})\cap R^n=\oplus_{j=1}^{k}(R(x-1)d_j)^{p^e}h_j=(x-1)^{p^e}V_0.
\end{aligned}
\end{equation*}
\end{proof}
%It is interesting and perhaps open question to describe all subgroups in $\mathcal{L}_{n,p}$ (and other groups), closed in the profinite or pro-$p$ topology. They are intersections of subnormal chains of subgroups of finite index.

%Even in the case of $\mathcal{L}=\mathcal{L}_{1,2}$ this is perhaps not yet done. But at least we have control of a part of lattice of power-$2$ subgroups of $\mathcal{L}$. Namely, $\mathcal{L}$ was realized in \cite{MR1866850} as a group generated by a $2$-state automaton over the alphabet $\{0,1\}$. The associated action of $\mathcal{L}$ on the binary rooted tree leads to 

\subsection{Automaton that generates $\mathcal{L}_{1,p}$.}
In this subsection we give description of a part of lattice of subgroups in $\mathcal{L}_{1,p}$ in terms of action on a $p$-regular rooted tree. We also provide automaton presentation of $\mathcal{L}_{1,p}$ showing, in particular, that groups $\mathcal{L}_{1,p}$ are self-similar.

Consider in the group $\mathcal{L}_1=\mathcal{L}_{1,p}$ the $p$-chain $H_m=(1-x)^mR\rtimes\bz$, for $m\geq 0$, that is $H_m$ has the triple $(1,(1-x)^mR,0)$. Since this is $p$-chain the index of $H_m$ is $p^m$. By the Lemma~\ref{normal_group} each $H_m$ is normal, in particular the chain is subnormal. Identify the $p$-regular rooted tree $T_p$ with the set of words over the alphabet $\Omega=\{0,1,\dots,p-1\}$. For each $\omega\in\Omega^*$, $\omega=\omega_0\cdots \omega_{m-1}$ construct the element $\bar{\omega}=\sum_{i=0}^{m-1}\omega_i(x-1)^i\in R$. The following Lemma holds:
\begin{lemma}
$\omega\mapsto (\bar{\omega},0)H_{|\omega|}$ is the isomorphism of $T_p$ and the coset tree of the chain $\{H_m\}$, where $|\omega|$  is the length of the word $\omega$.
\end{lemma} 
\begin{proof}
For injectivity, consider that $(\bar{\omega},0)H_{|\omega|}=(\bar{\theta},0)H_{|\theta|}$ if and only if $|\theta|=|\omega|=m$ and $\bar{\theta}-\bar{\omega}\in (1-x)^mR$, that is \Iff $\bar{\omega}=\bar{\theta}$, which means that $\omega_i=\theta_i$ and hence $\omega=\theta$. For surjectivity, consider that for any $(w,s)\in\mathcal{L}_1$ we have that $(w,s)H_m=(w,0)H_m$. Note also that $R/(1-x)^mR=\mathbb{F}_p[x]/\langle(1-x)^m\rangle$, so for any $w\in R$ there is a sequence $w_i\in\Omega$, such that 
\[
w=\sum_{i=0}^{m-1}w_i(x-1)^i\mod (1-x)^mR,
\]
for any $m\geq 0$. 
Finally, it is easy to establish that the map $\omega\mapsto\bar{\omega}$ preserve the property of vertices being adjoint.
\end{proof}
In the following Proposition we present the description of the stabilizers of the action of $\mathcal{L}_1$ on the coset tree of $\{H_m\}$, identified with $T_p$ with the help of the map from the above lemma:
\begin{proposition}\label{onne}
The subgroup $\Stab_{\mathcal{L}_1}(\omega)$ has the triple $(1,(1-x)^{|\omega|}R,(1-x)\bar{\omega})$, and the stabilizer of the $m$-th level $\Stab_{\mathcal{L}_1}(m)$ has the triple $(p^e,(1-x)^mR,0)$, where $e$ is smallest such that $m\leq p^e$.
\end{proposition}
\begin{proof}
Note that in the action on the coset tree the stabilizer of the vertex $gH_m$ is $gH_mg^{-1}$. Hence $\Stab_{\mathcal{L}_1}(\omega)=(\bar{\omega},0)H_{|\omega|}(-\bar{\omega},0)$, from which the formula for the triple follows. 

To find $\Stab_{\mathcal{L}_1}(m)=\cap_{\omega:|\omega|=m}\Stab_{\mathcal{L}_1}(\omega)$ use Corollary~\ref{3_2}. It follows that $\Stab_{\mathcal{L}_1}(m)$ has the triple $(r,(1-x)^mR,v)$ where $r>0$ is minimal such that $(1-x^r)\bar{\omega}=(1-x^r)\bar{\omega'}\mod (1-x)^mR$ for any $\omega,\omega'$ of length $m$, and $v=(1-x^r)\bar{\omega}$. Choosing $\omega=0^m$ we notice that we can take $v=0$ and $r>0$ is minimal such that $(1-x^r)\bar{\omega'}=0\mod (1-x)^mR$ for all $\omega'$ of length $m$. Equivalently, $r>0$ is minimal such that $(1-x)^m$ divides $1-x^r$, that is $x^r=1+(x-1)^mf(x)$ for some $f\in\mathbb{F}_p[x]$. Let $r=p^er'$ and change variable $x$ to $y=x-1$, then we have that $\dots+r'y^{p^e}+1=(y^{p^e}+1)^{r'}=(y+1)^r=1+y^mf(y+1)$, and so since $r'\neq 0$ in $\mathbb{F}_p$ one must have that $p^e\geq m$, hence the minimal such $r=p^e$ with $e\geq\log_p(m)$. 
\end{proof}
Now we will describe the stabilizers of infinite paths. To do this, we need a little preparation. By the proof of the Corollary~\ref{need} $\varprojlim_j R/(1-x)^mR=\mathbb{F}_p[[t]]$, and the subalgebra of $\mathbb{F}_p[[t]]$ generated by $t$ and $(t-1)^{-1}=-\sum_{i\geq 0}t^i$ is isomorphic to $R$ by the isomorphism that maps $t-1$ to $x$, so we may consider $R$ as a subalgebra of $\mathbb{F}_p[[t]]$. Then any $\omega\in\Omega^{\mathbb{N}}$ defines an element $\bar{\omega}\in\mathbb{F}_p[[t]]$ by the rule $\bar{\omega}=\sum_{i\geq 0}\omega_i t^i$. Let $\mathbb{F}_p((t))$ be the ring of fractions of $\mathbb{F}_p[[t]]$, then we can consider $(1-x^s)^{-1}R$ as a subset of it.
\begin{proposition}
For $\omega\in\Omega^{\mathbb{N}}$ the $\Stab_{\mathcal{L}_1}(\omega)$ is the cyclic group generated by $(w,s)$ in case $s>0$ is the smallest integer such that $\bar{\omega}\in (1-x^s)^{-1}R$, and $w=(1-x^s)\bar{\omega}$. In case there is no such $s$ the $Stab_{\mathcal{L}_1}(\omega)$ is trivial.
\end{proposition}
 \begin{proof}
Denote by $\omega^{(m)}$ prefix of length $m$ of $\omega$. Note that $\cap_m(1-x)^mR$ is trivial, so  $\Stab_{\mathcal{L}_1}(\omega)=\cap_m\Stab_{\mathcal{L}_1}(\omega^{(m)})$ can be either trivial or generated by some $(w,s)$ for $s>0$. The second case happens if and only if for every $m$ we have $(w,s)\in\Stab_{\mathcal{L}_1}(\omega^{(m)})$, and $s$ is smallest that this is possible. Since by Proposition~\ref{onne} the triple of $\Stab_{\mathcal{L}_1}(\omega^{(m)})$ is $(1,(1-x)^mR,\sum_0^{m-1}\omega_i(x-1)^{i+1})$, and 
 \[
 (\sum_0^{m-1}\omega_i(x-1)^{i+1},1)^s=((1-x^s)(\sum_0^{m-1}\omega_i(x-1)^{i}),s),
 \]
  this can happen if and only if $w=(1-x^s)(\sum_0^{m-1}\omega_i(x-1)^{i})\mod (1-x)^mR$ for every $m$, or equivalently that $w-(1-x^s)\bar{\omega}\in(1-x)^mR$. Since $\cap_m(1-x)^mR$ is trivial, it means that $w=(1-x^s)\bar{\omega}$.
 \end{proof}
We finally describe the automaton that generates the action of $\mathcal{L}_{1,p}$ on the tree.
\begin{proposition}
Let $b=(e,0)$ and $a=(0,1)$ be the generators of $\mathcal{L}_1$. Then their action on vertices of the tree can be described as  
\begin{equation}\label{starr}
\begin{aligned}
& b(\omega_0\omega_1\dots\omega_{m})=(\omega_0+1)\omega_1\dots\omega_m, \\&
a(\omega_0\omega_1\dots\omega_{m})=\omega_0(\omega_1+\omega_0)\dots(\omega_{m}+\omega_{m-1}),
\end{aligned}
\end{equation}
where the addition is mod $p$.
The relations \ref{starr} lead to the following recursive formula $a=(a,ba,\dots,b^{p-1}a)$ describing the action of $a$ in terms of of its projections on subtrees with roots at first level. Denoting $a_i=b^ia$, $0\leq i\leq p-1$, we see that $\mathcal{L}_1$ is generated by the automaton with $p$ states $a_0,\dots, a_{p-1}$.
\end{proposition}
\begin{remark}
Note that for $p=2$ it is the automaton studied in \cite{MR1866850}.
\end{remark}
Here are Moore diagrams of automata generating $\mathcal{L}_p$ for $p=2$: 
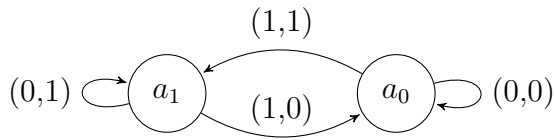
\begin{center}
\begin{tikzpicture}[->,>=stealth',shorten >=1pt,auto,node distance=2cm] 
    \node[state] (a_0) {$a_0$}; 
   \node[state] (a_1) [ left=of a_0] {$a_1$}; 
    \path[->] 
    (a_0) edge [bend right] node [swap] {(1,1)} (a_1)
               edge  [loop right] node {(0,0)} ()
    (a_1) edge [bend right] node {(1,0)}(a_0)
             edge  [loop left] node {(0,1)} ();
              \end{tikzpicture}
\end{center}

and $p=3$:
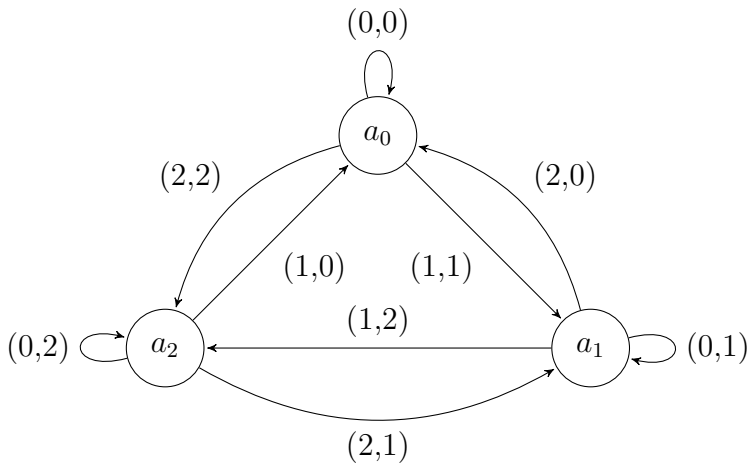
\begin{center}
\begin{tikzpicture}[->,>=stealth',shorten >=1pt,auto,node distance=4cm,on grid] 
    \node[state] (a_0) {$a_0$}; 
   \node[state] (a_1) [ below right=of a_0] {$a_1$}; 
   \node[state] (a_2) [ below left=of a_0] {$a_2$}; 
    \path[->] 
    (a_2) edge [bend right] node [swap] {(2,1)} (a_1)
               edge node [swap] {(1,0)} (a_0)
               edge  [loop left] node {(0,2)} ()
    (a_1) edge node [swap] {(1,2)}(a_2)
             edge [bend right] node [swap] {(2,0)}(a_0)
             edge  [loop right] node {(0,1)} ()
    (a_0) edge [bend right] node [swap] {(2,2)}(a_2)
             edge node [swap] {(1,1)}(a_1)
             edge  [loop above] node {(0,0)} ();
\end{tikzpicture}
\end{center}

\section{Application to relative gradient rank}\label{sec4}

As was already mentioned in the introduction, Lackenby defined an interesting group theoretical notion,
the rank gradient of a group, which happens to be useful in topology, studies on orbit equivalence, on amenability
properties of  groups and other topics  \cite{MR2151608}. % perevirte bud' laska chy ce ta stattia jaku Vy maly na uvazi
%\newline\newline
Given a group $G$ and a descending sequence $\lb H_i \rb_{i=1}^{\infty}$ of subgroups of finite index the
rank gradient of the sequence $\lb H_i \rb$ with respect to $G$ is defined as

$$RG(G,\lb H_i \rb ) = \lim_{i\to \infty} \frac{d(H_i)-1}{[G:H_i]},$$
where $d(H)$ denotes the minimal number of generators of a group $H$.
%\newline\newline

The notion of amenable groups was introduced by J. von Neumann in 1929 and plays important role in many
areas of mathematics \cite{MR1645068}. There are a number
of results due to Lackenby, Abert, Jaikin-Zapirain, Luck and Nikolov, showing that
amenability of $G$ or of a normal subgroups of $G$  implies vanishing of the rank gradient. For instance,
finitely generated infinite amenable groups have zero gradient rank  with respect to any normal chain with
trivial intersection (see Theorem 5 in \cite{AJZN}).
%\newline\newline
Obviously it is reasonable to restrict  study of the rank gradient for sequences $\lb H_i \rb$ with trivial
core (i.e. no nontrivial normal subgroups in the intersection $\cap_i H_i$).  The attention was  mostly
to the case when $\cap_i H_i = \lb 1 \rb$  but  the case $\cap_i H_i \neq \lb 1 \rb$ is also interesting. 

\begin{question}\label{q2}
\cite{AJZN} Let $G$ be a finitely generated infinite amenable group.  Is it true that $RG(G,\lb H_i \rb)=0$
for any chain with trivial intersection?\\
One can also ask whether the following stronger statement is true:
Is it true that $RG(G,\lb H_i \rb)=0$ if the chain has trivial core?
\end{question}

In case of  groups $\mathcal{L}_n$ the answer to Question~\ref{q2} is positive:
\begin{theorem}\label{rrr}
Let $H_i$, $i\geq 1$ be a descending chain of subgroups in $\mathcal{L}_n$ with trivial core, then $RG(\mathcal{L}_n,\{H_i\})=0.$ If $n=1$ the converse is also true. 
\end{theorem}
\begin{proof}
Let $(s_i,V_i,v_i)$ be the triples of $H_i$, and let $m_i=\dim_{\mathbb{F}_p}\mathcal{A}_n/V_i$. By the Theorem~\ref{3_3} $H_i$ is isomorphic to $\mathcal{L}_{ns_i}$, thus $d(H_i)=ns_i+1$. On the other hand $[\mathcal{L}_n:H_i]=s_ip^{m_i}$, thus $RG(\mathcal{L}_n,\{H_i\})=0$ if and only if $m_i\rightarrow\infty$ when $i\rightarrow\infty$. Suppose $m_i\not\rightarrow\infty$. Since the chain is descending, $m_i$ is nondecreasing, therefore the sequence $\{m_i\}$ stabilizes, that is, there is such $i_0$ that $V_{i_0}=V_{j}$ for $j\geq i_0$, and let us denote it $V=V_{i_0}$. Let $W=\cap_{t\in\mathbb{Z}}x^tV$, then $xW=W$ and thus $W$ is a normal subgroup of $\mathcal{L}_n$ that is contained in all $H_i$, $i\geq 1$. Note that  $W$ actually equals to the finite intersection $\cap_{t=1}^{t=s_0}x^tV$ since by the Lemma~\ref{normal_group} $x^{s_0}V=V$. Therefore $W$ has finite index in $\mathcal{A}_n$ and thus is not trivial. It follows that the core of $\{H_i\}$ is also not trivial, since it contains $W$.   

Let $n=1$. Above we showed that $RG(\mathcal{L}_n,\{H_i\})\neq 0$ if and only if $V_i=H_i\cap\mathcal{A}_n$ stabilizes. Let $N$ be the core of $\{H_i\}$ and suppose $N$ is not trivial. Then it follows from Corollary~\ref{one_more_cor} that $N'=N\cap\mathcal{A}$ has finite index in $\mathcal{A}_1$.  Since $N'$ is contained in each $V_i$, the chain $\{V_i\}$ must stabilize.
\end{proof}

If $\cap_{i=1}^{\infty} H_i = H  $ then $H$ is a closed subgroup with respect to the profinite topology and
$RG(G,\lb H_i \rb)$ is a characteristic of the pair $(G,H)$ which in some situations may characterize the
pair $(G,H)$ up to isomorphism (two pairs $(G,H),(P,Q)$ are isomorphic if there is an isomorphism $\phi: G
\to P$ such that $\phi(H) = Q$).
%\newline\newline
If $RG(G,\lb H_i \rb)=0$ then the function
$$rg(m) = rg_{(G,\lb H_i \rb)}(m) = \frac{d(H_m)-1}{[G:H_m]},$$
which we call  \emph{relative rank gradient (function)} can be a subject of investigation
 (we  omit the index $(G,\lb H_i \rb)$ if the group and chain in consideration are understood).
 The rate of decay of $rg(m)$ is a sort of characteristic of the pair $(G,H)$ and may characterize the way $H$
is embedded in $G$ as a subgroup.
%\newline\newline
Note that the same subgroup can be obtained as  intersection of distinct chains. For instance, if the chain
consists of subgroups isomorphic to  the enveloping group (and such chains exist in scale invariant groups
for instance) one can delete certain elements in $H_i$ thereby allowing $rg(m)$ to decay as fast as one would
like.  Therefore, it is reasonable to restrict  attention  to the case when for some prime $p$ the index
$[H_{i}:H_{i+1}] = p$. In this case we say  chain is a \emph{$p$-chain}. One of corollaries of our results
listed above  is  showing that a lamplighter group contains subnormal $p$-chains of trivial intersection, with distinct rates of decay of the
rank gradient function.

Let $H_i$, $i\geq 0$ be a descending $p$-chain of subgroups of $\mathcal{L}_n$ such that $H_0=\mathcal{L}_n$. Note that for each $i$ the subgroup $H_i$ is isomorphic to $\mathcal{L}_k$ for some $k$ by Theorem~\ref{3_3}. Then there are two possible cases for $H_{i+1}$: either its projection on $\mathbb{Z}$ (under the identification of $H_i$ and $\mathcal{L}_k$) is $\mathbb{Z}$ or it is $p\mathbb{Z}$. In the first case $d(H_{i+1})=d(H_i)$ and in the second case $d(H_{i+1})=pd(H_i)-p+1$. We have the following result
  \begin{theorem}\label{4.2}
 Suppose that $g:\mathbb{N}\rightarrow\mathbb{N}$ is such that $g(0)=p+1$ and for each $i\in\mathbb{N}$ either $g(i+1)=g(i)$ or $g(i+1)=pg(i)-p+1$, and suppose that the set $\{i|g(i)=g(i+1)\}$ is infinite. Then there is a descending subnormal $p$-chain $\{H_i\}$ of subgroups of $\mathcal{L}_1$, such that $H_0=\mathcal{L}_1$, $d(H_i)=g(i)$ and $\cap_i H_i=\{1\}$.
 \end{theorem}
\begin{proof}
Note that by Theorem~\ref{max} we can construct some descending subnormal $p$-chain $\{H'_i\}$ such that all of the above properties hold except possibly for $\cap_i H'_i=\{1\}$. Note also that 
\begin{equation*}
\begin{aligned}
& \frac{g(i+1)-1}{p^{i+1}}=\frac{g(i)-1}{p^{i}}\text{ if }g(i+1)=pg(i)-p+1,\\&
\frac{g(i+1)-1}{p^{i+1}}=\frac{1}{p}\frac{g(i)-1}{p^{i}}\text{ if }g(i+1)=g(i).
\end{aligned}
\end{equation*}
Since $\{H'_i\}$ is a $p$-chain, $[\mathcal{L}_1:H_i]=p^{i}$, and the fact that the set $\{i|g(i+1)=g(i)\}$ is infinite implies that $RG(\mathcal{L}_1,\{H'_i\})=0$. By the Theorem~\ref{rrr} it follows that the core of the chain $\{H'_i\}$ is trivial. 

Consider then the rooted binary tree associated to this chain, and the action of $G$ on it by multiplication from the right. The kernel of this action is the core of the chain $\{H'_i\}$, and therefore the action is faithful. It is clear that it is also spherically transitive. Then it follows from the Corollary~\ref{hered} and the Proposition~4.11 in \cite{Grig} that the action of $\mathcal{L}_1$ on the boundary of this tree is topologically free, thus the set of infinite paths in the tree with nontrivial stabilizer is a union of nowhere dense subsets of the boundary of the tree. In particular its complement is not empty. Take a point from this complement, that is an infinite path $\{g_iH'_i\}$ (for some $g_i\in\mathcal{L}_1$) with trivial stabilizer. It is easy to see that this stabilizer is the intersection of descending subnormal $p$-chain $H_i=g_iH'_ig_i^{-1}$. To finish the proof it is left to note that since $H'_i$ and $H_i$ are isomorphic,  $d(H_i)=d(H'_i)=g(i)$ for all $i\geq 0$.  
\end{proof}

%\section{Construction of chains}   Perepusatu cyu chastunu zgidno mogo shoino napusanogo  zauvagennya .  Napustau ce yak dovedennya teormu.

%Let $\mathcal{L}_n=\mathbb{Z}^n_2\wr\mathbb{Z}$, and denote $\mathcal{L}_1$ by simply $\mathcal{L}$. It
%follows from the previous lemmas that $\mathcal{L}$ contains $3$ subgroups of index $2$ (two of them are
%isomorphic to $\mathcal{L}$ and one to $\mathcal{L}_2$). $\mathcal{L}_2$ has $7$ subgroups of index $2$, $3$
%of them isomorphic to $\mathcal{L}_4$ and $4$ to $\mathcal{L}_2$, etc.

%Having a subgroup $H<\mathcal{L}$ of index $2^k$, which is isomorphic to $\mathcal{L}_{2^i}$ for some
%$i\leq{}k$, we can take a subgroup of index $2$ isomorphic to $\mathcal{L}_{2^i}$ or to
%$\mathcal{L}_{2^{i+1}}$. Call the first choice -- type $0$ and the second --  type $1$. It is obvious that
%$d(\mathcal{L}_n)=n+1$, thus two possibilities of choice of a subgroup of index $2$ up to a type allow, given
%a sequence $w\in\{0,1\}^\mathbb{N}$ construct a chain $\{H^w_n\}$ such that $H_n$ is a subgroup of index $2$
%in $H_{n-1}$ of type $w_n$. It is clear that in such a way we get uncountably many different groups $H^w$
%({\bf $H^w=\cap{}H^w_n$?}) with pairwise different ({\bf nonequivalent?}) $rg^w(n)$ -- descending rank
%gradient functions.

%\begin{remark}
%If $r=\lim_{n\rightarrow\infty}rg^w(n)>0$ then $r=1/2^k$ for some $k$ and $w$ has only $k$ symbols $0$. In
%this case $H^w=\cap_n{}H^w_n$ contains nontrivial normal subgroup.
%\end{remark}

%\section{Concluding remarks}.

Being solvable groups $\mathcal{L}_n$ are elementary amenable groups.   The $p$-groups ($p$ prime) of
intermediate growth constructed in \cite{MR781246} are nonelementary amenable groups. Since for them
gradient rank is also zero it is reasonable to study relative gradient rank for various closed subgroups in
pro-$p$-topology and $p$-chains representing them.

One of such groups is the  group $\mathcal{G}=\langle a,b,c,d\rangle$ generated by 4 involutions \cite{MR565099}. This group (as
well as most of other known groups of intermediate growth) act faithfully by automorphisms on a regular
rooted tree (in this case on a binary tree). We identify in a standard way vertices of binary tree
with binary words. For the stabilizer $P=st_{\mathcal G}(1^\infty)$ of infinite ray  $1^\infty$ (point of the
boundary $\partial T$ of a tree) the corresponding sequence is the sequence  $P_m=st_{\mathcal G}(1^m)$.  As
is shown in \cite{}, $P_1=st_{\mathcal{G}}(1)=\langle c,d,c^a,d^a\rangle$, $P_m=R_m\rtimes\{1,b,c,d\}$ for $m\geq 2$, where $R_m$ are iterative semidirect products
$R_1=K=\langle (ab)^2\rangle^{\mathcal{G}}$, and $R_m=(K\times R_{m-1})\rtimes\{1,(ac)^4\}$ for $m\geq 2$.
 %Moge provestu bilshe detalei, vzyavshu ih iz moei statti z Bartholdi,  journal Serdica 2002
Computations then show that $d(P_m)=m+4$ for $m\geq 2$, $d(P_1)=4$.
 This implies that
\[
rg_{(\mathcal{G},\{P_i\})}(m)=\frac{m+4}{2^m},
\]
for $m\geq 2$.

%Another interesting sequence is the sequence of stabilizers  $st_{\mathcal G}(n)$ of levels which are normal
%subgroups and intersect trivially. For them  the index is =....   \cite{}  and  they are products of $2^{n-3}$
%??   copies of group $st(3)$

%vzyatu cyu informaciyu is Serdica  i moei statti Solved and Unsolved ...  2005

%Therefore  $rg(n) \sim ???$  in this case.    We see that in both cases  the decay is the same  and is ????
%(perevirutu chu ce tak,  ya smutno pamyatayu)
\begin{question}
Which subgroups of $\mathcal{G}$ are closed in profinite topology and what kind of asymptotic behavior $rg(m)$
holds for them?
\end{question}

\bibliography{bibliography}{}
\bibliographystyle{plain}
\end{document}